\documentclass[11pt,oneside,reqno]{amsart}
\usepackage[utf8x]{inputenc}

\usepackage{textcomp}
\usepackage{amsmath,amssymb}
\usepackage{mathrsfs}
\usepackage[T2A]{fontenc}
\usepackage{amsthm}
\usepackage{epigraph}
\usepackage{array}
\usepackage[linktocpage]{hyperref}
\usepackage{amsfonts,graphics,amsthm,amsfonts,amscd,latexsym}
\usepackage{epsfig}
\usepackage{flafter}
\usepackage{mathtools}
\usepackage{comment}
\usepackage{stmaryrd}
\usepackage{enumitem}
\usepackage{textcomp}
\usepackage{epigraph}
\usepackage{longtable}
\usepackage{tikz-cd}
\hypersetup{
colorlinks=true,
linkcolor=red,
citecolor=green,
filecolor=blue,
urlcolor=blue
}
\usepackage{tikz}
\usetikzlibrary{quotes,babel}

\usetikzlibrary{graphs,positioning,arrows,shapes.misc,decorations.pathmorphing}

\tikzset{
>=stealth,
every picture/.style={thick},
graphs/every graph/.style={empty nodes},
}

\tikzstyle{vertex}=[
draw,
circle,
fill=black,
inner sep=1pt,
minimum width=5pt,
]
\usepackage[position=top]{subfig}
\usepackage[alphabetic]{amsrefs}
\usepackage{amssymb}
\usepackage{color}

\calclayout

\usetikzlibrary{decorations.pathmorphing}
\tikzstyle{printersafe}=[decoration={snake,amplitude=0pt}]

\newcommand{\oo}{\mathcal{O}}

\newcommand{\pp}{\mathbb{P}}

\newcommand{\coreg}{\mathrm{coreg}\,}
\newcommand{\coregG}{\mathrm{coreg}_G\,}

\newcommand{\CG}{\ensuremath{\mathfrak{C}}}
\newcommand{\DG}{\ensuremath{\mathfrak{D}}}
\newcommand{\AG}{\ensuremath{\mathfrak{A}}}
\newcommand{\SG}{\ensuremath{\mathfrak{S}}}

\newcounter{theoremsec}[section]

\def\O#1.{\mathcal {O}_{#1}}
\def\pr #1.{\mathbb P^{#1}}
\def\af #1.{\mathbb A^{#1}}
\def\ses#1.#2.#3.{0\to #1\to #2\to #3 \to 0}
\def\xrar#1.{\xrightarrow{#1}}
\def\K#1.{K_{#1}}
\def\bA#1.{\mathbf{A}_{#1}}
\def\bM#1.{\mathbf{M}_{#1}}
\def\bL#1.{\mathbf{L}_{#1}}
\def\bB#1.{\mathbf{B}_{#1}}
\def\bK#1.{\mathbf{K}_{#1}}
\def\subs#1.{_{#1}}
\def\sups#1.{^{#1}}

\usepackage{tikz}
\usetikzlibrary{matrix,arrows,decorations.pathmorphing}

\newtheorem{theorem}{Theorem}[section]

\newtheorem{lemma}[theorem]{Lemma}

\newtheorem{proposition}[theorem]{Proposition}

\newtheorem{corollary}[theorem]{Corollary}

\theoremstyle{definition}
\newtheorem{definition}[theorem]{Definition}

\newtheorem{example}[theorem]{Example}

\newtheorem{notation}[theorem]{Notation}
\newtheorem{correspondence}[theorem]{Canonical correspondence}

\newtheorem{problem}[theorem]{Problem}
\newtheorem{question}[theorem]{Question}

\newtheorem{remark}[theorem]{Remark}

\theoremstyle{remark}

\numberwithin{equation}{section}

\usepackage[all]{xy}
\newcounter{rownumber}[figure]
\newcounter{rownumber-irr}[figure]
\newcounter{rownumber-p1}[figure]
\begin{document}

\title{On the coregularity of del Pezzo surfaces with du Val singularities}

\author[K.~Loginov]{Konstantin Loginov}

\author[A.~Trepalin]{Andrey Trepalin}

\address{\emph{Konstantin Loginov}
\newline
\textnormal{Steklov Mathematical Institute of Russian Academy of Sciences, Moscow, Russia.}
\newline
\textnormal{Laboratory of AGHA, Moscow Institute of Physics and Technology.}
\newline
\textnormal{\texttt{loginov@mi-ras.ru}}}

\thanks{
        \newline
		}
		
\address{\emph{Andrey Trepalin}
\newline
\textnormal{Steklov Mathematical Institute of Russian Academy of Sciences, Moscow, Russia.}
\newline
\textnormal{\texttt{trepalin@mccme.ru}}}

\setcounter{tocdepth}{1}

\begin{abstract}
We compute the coregularity of del Pezzo surfaces with du Val singularities. To this aim, we study the relation between del Pezzo surfaces of degree $1$ and elliptic fibrations. 
It~turns out that del Pezzo surfaces with positive coregularity correspond to isotrivial elliptic fibrations with some special properties. 
We also prove results about coregularity of del Pezzo surfaces over non-algebraically closed fields of characteristic $0$. Our results confirm the expectation that ``most'' del Pezzo surfaces have coregularity $0$, while del Pezzo surfaces with positive coregularity enjoy some special properties.  
\end{abstract}

\maketitle
\tableofcontents

\section{Introduction}

We work over an algebraically closed field $\mathbb{K}$ of characteristic $0$ unless stated otherwise.
A normal projective variety $X$ is called a Fano variety if its anti-canonical class $-K_X$ is an ample $\mathbb{Q}$-Cartier divisor. 
A natural way to study Fano varieties is by looking at its (pluri-)anti-canonical linear system $|-lK_X|$ for $l\geqslant 1$. In dimension not greater than $3$ the existence of a smooth element in the linear system $|-K_X|$ was established by V. Shokurov \cite{Sh80}, which led to an inductive approach to the classification of smooth Fano threefolds. 
It is also interesting to look at singular elements in~$|-lK_X|$. 
It turns out that the singularities of such elements are related to the variation of smooth elements of these linear systems in terms of their moduli. 

To measure ``how singular'' the elements of $|-K_X|$ could be, the following invariants were introduced. The \emph{dual complex} of a reduced simple normal crossing divisor $D$ on a smooth variety is a~topological space that measures the combinatorial complexity of this divisor. Using the resolution of singularities, this definition can be generalized to the case of log canonical pairs.

The notion of \emph{regularity} $\mathrm{reg}(X)$ of a projective variety $X$ was introduced to V. Shokurov in \cite{Sh00}. By definition, it is the maximum of dimensions of dual complexes of log canonical complements on a given Fano variety, see Definition \ref{defin-regularity}. One can say that the regularity measures how far a given Fano variety is from being \emph{exceptional}. The latter means that any log canonical complement is in fact Kawamata log terminal. 
The dual notion of \emph{coregularity} 
\[
\mathrm{coreg}(X)=\dim X-\mathrm{reg}(X)-1, 
\]
was defined in \cite{Mo24}. 
The coregularity of a variety of dimension $n$ is an integer in the set $\{0,\dots,n\}$. To give an example, the sum of all torus invariant divisors on a toric variety gives a $1$-complement which shows that the coregularity of a toric variety is equal to $0$.

The study of regularity of Fano varieties drew a lot of attention recently, see \cite{Mo24}, \cite{FFMP25}, \cite{LPT25} and references therein. 
In \cite{ALP24} the coregularity of smooth $3$-dimensional Fano varieties was computed. 
It turns out that ``most'' smooth Fano threefolds have coregularity $0$. 
From~\cite{Zha25} it~follows that in any deformation family of smooth $3$-dimensional Fano varieties there exists an~element of coregularity $0$. In the case of del Pezzo surfaces, there is the following simple result.


\begin{proposition}[{\cite[Proposition 2.3]{ALP24}}]
\label{prop-smooth-dP}
Let $\mathbb{K}$ be an algebraically closed field of characteristic~$0$. 
Let $X$ be a smooth del Pezzo surface of degree $d=(-K_X)^2$ defined over $\mathbb{K}$.
If $d\geqslant 2$ then $X$ has coregularity $0$.
A general smooth del Pezzo surface with $d=1$ has coregularity $0$. For a special smooth del Pezzo surface $X$ of degree $1$, one has $\mathrm{coreg}(X) = 1$. 
\end{proposition}

The word special in Proposition \ref{prop-smooth-dP} means that the linear system $|-K_X|$ does not contain a~nodal curve, so that all the singular elements in $|-K_X|$ are cuspidal curves (cf. Lemma \ref{lem-red-irred-fibers}). 
This condition can be rephrased in a different way which allows us to generalize Proposition \ref{prop-smooth-dP} to~the~case of mildly singular del Pezzo surfaces, that is, del Pezzo surfaces with du Val singularities. 

To any degree $1$ del Pezzo surface with at worst du Val singularities one can canonically associate a relatively minimal elliptic fibration $\phi_Y\colon Y\to \mathbb{P}^1$. To do this, one needs to consider the blow-up of the unique basepoint of the linear system $|-K_X|$ followed by the minimal resolution. There exists a~correspondence between singular elements in $|-K_X|$ and singular fibers of the elliptic fibration~$\phi_Y$, which we recall below. Using this correspondence Proposition \ref{prop-smooth-dP} can be reformulated as follows (recall that an elliptic fibration is called \emph{isotrivial} if all its smooth fibers are isomorphic):

\begin{proposition}
\label{prop-smooth-dp-reformulated}
Let $\mathbb{K}$ be an algebraically closed field of characteristic~$0$. 
Let $X$ be a smooth del Pezzo surface $X$ defined over $\mathbb{K}$. Then either $\coreg (X)=0$, or the following equivalent conditions hold: 
\begin{itemize}
\item 
$(-K_X)^2=1$ and the elliptic fibration $\phi_Y$ is isotrivial, 
\item 
all smooth elements in the linear system $|-K_X|$ are isomorphic,
\item
$\coreg (X)=1$.
\end{itemize}
\end{proposition}

Proposition \ref{prop-smooth-dp-reformulated} shows that the coregularity of a Fano variety $X$ could be related to the variation of smooth elements in $|-K_X|$.
Recall that for a del Pezzo surface $X$ with $(-K_X)^2\geqslant 3$ we~have that $-K_X$ is very ample and defines an embedding $X\subset \mathbb{P}^d$ where $d=(-K_X)^2$. 
Hence, from Proposition~\ref{prop-smooth-dp-reformulated} we see that the problem of computing coregularity is related to the study of~projective varieties such that all its smooth hyperplane sections are isomorphic. The latter is a~classical problem, see e.g. \cite{Beau90} in the case of hypersurfaces. 
Our first main result is as follows (we~denote by $\mathrm{Sing}(X)$ the set of all singular points of $X$):

\begin{theorem}
\label{thm-dp1-main}
Let $\mathbb{K}$ be an algebraically closed field of characteristic~$0$. 
Let $X$ be a del Pezzo surface with du Val singularities of degree $d=(-K_X)^2$ defined over $\mathbb{K}$.
If $d\geqslant 2$ then
$$
\mathrm{coreg}(X)=\coreg_1(X)=0.
$$
If $d=1$, denote the associated elliptic fibration by $\phi_Y\colon Y\to \mathbb{P}^1$. Then
one has $\mathrm{coreg}(X)=0$ if and only if one of the following holds: 
\begin{itemize}
\item
$\phi_Y$ is not isotrivial, 
\item
 $\phi_Y$ is isotrivial, and
\[
\mathrm{Sing}(X)\in \{2D_4, D_4+2A_1\}.
\]
\end{itemize}
Otherwise, $\mathrm{coreg}(X)=1$.
\end{theorem}

Theorem \ref{thm-dp1-main} agrees with the expectation that ``most'' Fano varieties have coregularity $0$. 
For the notion of first coregularity $\mathrm{coreg}_1(X)$ see Definition \ref{defin-regularity}. 
We also prove a version of Theorem \ref{thm-dp1-main} which is formulated in terms of the configuration of singularities $\mathrm{Sing}(X)$. 

\begin{theorem}
\label{main-thm2}
Let $\mathbb{K}$ be an algebraically closed field of characteristic~$0$. 
Let $X$ be a del Pezzo surface with du Val singularities of degree $1$ defined over $\mathbb{K}$.
Then 
\begin{itemize}
\item 
either $\mathrm{coreg}(X)=1$, and $X$ has one of the following configurations of singular points:
\begin{itemize}
\item
 $\mathrm{Sing}(X)\in \{E_8, E_7+A_1, E_6+A_2\}$, in which case $\rho(X)=1$, or 
\item 
$\mathrm{Sing}(X)\in \{E_6, D_4+A_2, D_4, 3A_2, 2A_2, A_2, 4A_1, \varnothing\}$, in which case $\rho(X)>1$,
\end{itemize}
\item
or $\mathrm{coreg}(X)=0$.
\end{itemize}
\end{theorem}

We point out that the surfaces with the configurations of singular fibers described in Theorem~\ref{main-thm2} also could have coregularity $0$. We have examples of such surfaces with Picard number $1$, see Remarks~\ref{rem-equations-j-0-bad} and \ref{rem-equations-j-1728-bad}. On the other hand, we expect that such surfaces with coregularity $0$ also exist in the case when the Picard number is greater than $1$. We also note that ``most'' du Val del Pezzo surfaces have coregularity $0$. For example, in the case of Picard number $1$, there exists $27$ possible configurations of singular points on du Val del Pezzo surface of any degree, see \cite{Ye02}, and~only for $4$ such configurations there exists a surface $X$ with $\mathrm{coreg}(X)=1$. 

We note that on a del Pezzo surface $X$ with du Val singularities of degree $1$, an element of~$|-K_X|$ is reduced, irreducible, and it can contain only one singular point of $X$, cf. Lemma~\ref{lem-red-irred-fibers}. For the correspondence between singular points of $X$ and singular fibers of $\phi_Y$ in this case, see Table~1. Note that this correspondence is not bijective, see Remark \ref{rem-notbijection}. However, using this correspondence we can reformulate 
Theorem \ref{main-thm2} in terms of singular fibers of the elliptic fibration $\phi_Y$. We use the notation of the Kodaira classification of singular fibers of elliptic fibrations. 
By $j$ we denote $j$-invariant of the fibers of $\phi_Y$ which is a function on the base $\mathbb{P}^1$.

\begin{theorem}
\label{cor-main}
Let $\mathbb{K}$ be an algebraically closed field of characteristic~$0$. 
Let $X$ be a del Pezzo surface with du Val singularities of degree $1$ defined over $\mathbb{K}$ . 
Then $\mathrm{coreg}(X)=1$ if and only if the elliptic fibration $\phi_Y$ is isotrivial and has one of the following configurations of singular fibers:
\[
j=0:\;\; II+II^*,\;\; IV+IV^*,\;\; 2II+IV^*, \;\; II+IV+I_0^*, \;\; 3II+I_0^*, \;\; 3IV, \;\; 2II+2IV,\;\; 4II+IV,\;\; 6II;
\]
\[
j=1728:\quad III+III^*,\quad 4III.
\]
In particular, one has $\mathrm{coreg}(X)=1$ and $\rho(X)=1$ if and only if the singular fibers of $\phi_Y$ are one of the following:
\begin{equation}
\label{eq-extremal}
II+II^*,\quad \quad III+III^*, \quad \quad IV+IV^*.
\end{equation}
\end{theorem}

An elliptic surface with one of the  configurations \eqref{eq-extremal} is  \emph{extremal}, which means that its Mordell--Weil group is finite. Such surfaces enjoy many interesting properties, see \cite{MP86} and references therein. Also, according to Theorem \ref{cor-main}, del Pezzo surfaces with positive coregularity are special in the following sense. 
Their $j$-function on the base is equal either to $0$ or to $1728$, which corresponds to~elliptic curves with complex multiplication, see Proposition \ref{rem-cm-ell-curves}. 
The corresponding elliptic fibrations $\phi_Y\colon Y\to \mathbb{P}^1$ admit an automorphism which induces complex multiplication on the general fiber of~$\phi_Y$.

Recall that a Calabi--Yau pair $(X,D)$ where $\dim X=n$ admits a {\em toric model} if $(X,D)$ is crepant birationally equivalent to $(\pp^n,H_0+\dots+H_n)$, see Definition~\ref{def:crep-bir-isom} for the notion of crepant birational equivalence. We say that a Fano variety admits a toric model if it admits a 1-complement $D$ such that the pair $(X, D)$ has a toric model. Note that if a Fano variety $X$ has a toric model then $X$ must be rational. 
In~\cite[Lemma 1.13]{GHK15} it is proved that 
a rational surface $X$ admits a toric model if and only if $\mathrm{coreg}_1(X)=0$.  
Hence, using Proposition \ref{prop-smooth-dP} we obtain that a general smooth del Pezzo surface
admits a toric model.
In~\cite[Theorem 3.3]{EFM24} it is shown that the toric model exists for ``most'' del Pezzo surfaces with du Val singularities and of Picard rank $1$, see Remark \ref{rem-EFM24}. 
In other words, the authors computed $\coreg_1(X)$ for du Val del Pezzo surfaces of Picard rank $1$. 
We prove a generalization of their result to the case of del Pezzo surfaces with du Val singularities of arbitrary Picard number.

\begin{theorem}
\label{cor-toric-models}
Let $\mathbb{K}$ be an algebraically closed field of characteristic~$0$. 
Let $X$ be a del Pezzo surface with du Val singularities defined over $\mathbb{K}$. 
Then $\mathrm{coreg}_1(X)=1$ (which is equivalent to the non-existence of a toric model of $X$) if and only if $(-K_X)^2=1$, and one of the following holds: 
\begin{enumerate}
\item
either $\phi_Y$ is isotrivial,  
\item 
or $X$ is isomorphic to one of the surfaces
\begin{equation}
\label{eq-two-exceptions}
X'_1(D_5+A_1), \quad \text{or } \quad X'_1(D_6),
\end{equation}
\end{enumerate}
described in Example \ref{ex-very-special-surfaces}. 
The surfaces \eqref{eq-two-exceptions} are special with given configurations of singularities. Also, for the surfaces \eqref{eq-two-exceptions}, the corresponding elliptic fibration $\phi_Y$ is not isotrivial, $\mathrm{coreg}_1(X)=1$, and  $\mathrm{coreg}(X)=\mathrm{coreg}_2(X)=0$. 
\end{theorem}

Recall that a \emph{Looijenga pair} is an lc surface pair $(X, D)$ such that $X$ is smooth and rational, and $\mathrm{coreg}_1(X)=0$ is attained on the pair $(X, D)$. Such pairs are the object of intensive study, see \cite{Fr15} and references therein. One may say that Looijenga pairs  generalize smooth del Pezzo surfaces with a nodal anti-canonical section. 

For the results on the existence of toric models for smooth Fano threefolds see \cite{LMV24}. It~is proven that a general rational smooth Fano threefold admits a toric model. In \cite{Du22} it is conjectured that a $3$-dimensional rational Calabi–Yau pair $(X, D)$ of index one and coregularity~$0$ has a~toric model. In \cite{EFM24}, a higher-dimensional analog of this conjecture is proposed: if $(X, D)$ is a~rational $n$-dimensional Calabi--Yau pair of index one and coregularity $0$ such that $\mathcal{D}(X, D)\simeq S^{n-1}$, and any stratum of the boundary divisor $D$ is rational, then $(X, D)$ admits a toric model. It would be interesting to understand the relation between the existence of a toric model and the variation of the anti-canonical elements on a given variety.

In the case when a del Pezzo surface $X$ is defined over a field of characteristic $0$ which is not algebraically closed, the definition of coregularity still makes sense, see Remark \ref{rem-non-closed-field}. 
We prove the~following result. 

\begin{theorem}
\label{main-thrm-3}
Let $X$ be a du Val del Pezzo surfaces defined over a field $\mathbb{K}$ of characteristic $0$. Assume further that 
\begin{itemize}
\item
either $(-K_X)^2\geqslant 5$,
\item 
or $(-K_X)^2\geqslant 3$ and $X(\mathbb{K})\neq\varnothing$. 
\end{itemize}
Then $\coreg(X) = \mathrm{coreg}_1(X)=0$.
\end{theorem}

The paper is organized as follows. In Section \ref{sec-prelim}, we recall some standard definitions and basic results on the geometry on del Pezzo surfaces and the coregularity. In Section \ref{sec-dp-over-c} we consider del Pezzo surfaces with at worst du Val singularities defined over the field $\mathbb{C}$, and compute their coregularity. 
In Section \ref{sec-dp-over-k} we consider del Pezzo surfaces with at worst du Val singularities defined over fields which are not algebraically closed, and prove some results on the coregularity of such surfaces. 
Finally, in Section \ref{sec-examples-and-questions} we propose several open questions related to our results. 

\medskip

\textbf{Acknowledgements.} We thank Alexander Kuznetsov and Yuri Prokhorov for useful discussions. Also we are grateful to the anonymous referee for many helpful suggestions. 

\section{Preliminaries}
\label{sec-prelim}
In this section we work over an algebraically closed field $\mathbb{K}$ of characteristic $0$ unless stated otherwise.
We
use the language of the minimal model program (the MMP for short), see
e.g.~\cite{KM98}.

\subsection{Contractions} By a \emph{contraction} we mean a surjective
morphism $f\colon X \to Y$ of normal varieties such that $f_*\oo_X =
\oo_Y$. In particular, $f$ has connected fibers. A
\emph{fibration} is defined as a contraction $f\colon X\to Y$ such
that $\dim Y<\dim X$.

\subsection{Pairs and singularities} A \emph{pair} $(X, D)$ consists
of a normal variety $X$ and a boundary \mbox{$\mathbb{Q}$-divisor} $D$ with
coefficients in $[0, 1]$ such that $K_X + D$ is $\mathbb{Q}$-Cartier.
Let $\phi\colon W \to X$ be a
birational contraction, and let $(X,D)$ be a pair. Write $K_W +D_W = \phi^*(K_X +D)$.
The \emph{log discrepancy} of a~prime divisor $E$ on $W$ with respect
to $(X, D)$ is $1 - \mathrm{coeff}_E D_W$ and it is denoted by $a(E, X, D)$. We~say $(X, D)$ is log canonical or simply lc (resp. klt)(resp., canonical) if $a(E, D, B)$
is $\geqslant 0$ (resp.~$> 0$)(resp., $\geqslant 1$) for every $E$ and every birational contraction $\phi$.

Let $(X, D)$ be an lc pair.
Then the \emph{log-canonical threshold} of an effective $\mathbb{Q}$-Cartier $\mathbb{Q}$-divisor $B$ with respect to the pair $(X, D)$ is defined as
\[
\mathrm{lct} (X, D; B) = \mathrm{sup} \{ \lambda \in \mathbb{Q}\ |\ (X, D + \lambda B)\ \text{is lc} \}.
\]


\begin{definition}\label{def:crep-bir-isom}
Let $(X,D)$ and $(Y,D_Y)$ be two pairs.
We say that these two pairs
are {\em crepant birationally equivalent} if there exists a commutative diagram
\begin{equation}
\begin{tikzcd}
& (V, D_V) \ar[rd, "\psi"] \ar[dl, swap, "\phi"] & \ \\
(X, D) \ar[rr, dashed, "\alpha"] & & (Y, D_Y)
\end{tikzcd}
\end{equation}
where $\alpha$ is a birational map, $\phi$ and $\psi$ are birational contractions, $\phi_*(D_V)=D$, $\psi_*(D_V) = D_Y$, and 
\[
K_V + D_V=\phi^*(K_X+D) = \psi^*(K_Y + D_Y).
\]
Note that here $D_V$ is not necessarily a boundary since it may have negative coefficients.  
In the previous setting, we write $(X,D)\simeq_{\rm cbir} (Y,D_Y)$.
Given $(X,D)$ and $\alpha\colon X \dashrightarrow Y$ as above, there is a~unique divisor $D_Y$ for which $(X,D) \sim_{\rm cbir} (Y,D_Y)$. 
The divisor $D_Y$ is called the {\em crepant transform} of $D$ on $Y$.
\end{definition}

\subsection{Complements and Calabi--Yau pairs}
\label{sect-log-CY}
Let $(X, D)$ be a log canonical pair. 
We say $(X,D)$ is a \emph{Calabi--Yau pair} (or \emph{CY pair} for short) if $K_X + D
\sim_{\mathbb{Q}} 0$. In this case, $D$ is called a {\em $\mathbb{Q}$-complement} of $K_X$. If $N(K_X + D)\sim 0$ for some $N$, we say that $D$ is an {\em $N$-complement} of $K_X$. 

\subsection{Dual complex and coregularity}
Let $D=\sum D_i$ be a Cartier divisor on a smooth variety~$X$. Recall that $D$ has \emph{simple normal crossings} (snc for short), if all the components $D_i$ of $D$ are smooth, and any point in $D$ has an open neighborhood in the analytic topology that is analytically equivalent to the union of coordinate hyperplanes.

\begin{definition}
\label{def-dual-complex}
Let $D=\sum_{i=1}^r D_i$ be a simple normal crossing divisor on a smooth variety $X$.
The~\emph{dual complex}, denoted by $\mathcal{D}(D)$ is a regular CW-complex (more precisely, a $\Delta$-complex, so it admits a natural PL-structure) constructed as follows.  
\begin{itemize}
\item
The simplices $v_Z$ of $\mathcal{D}(D)$ are in bijection with irreducible components $Z$ of the intersection $\bigcap_{i\in I} D_i$ for any subset $I\subseteq \{ 1, \ldots, r\}$, and the dimension of $v_Z$ is equal to $\#I-1$.
\item
The gluing maps are constructed as follows. 
For any subset $I\subseteq \{ 1, \ldots, r\}$, let $Z\subset \bigcap_{i\in I} D_i$ be any irreducible component, and for any $j\in I$ let $W$ be a unique component of $\bigcap_{i\in I\setminus\{j\}} D_i$ containing $Z$. Then the gluing map is the inclusion of $v_W$ into $v_Z$ as a~face of $v_Z$ that does not contain the vertex $v_j$. 
\end{itemize}
\end{definition}

Note that the dimension of $\mathcal{D}(D)$ does not exceed $\dim X-1$. If $\mathcal{D}(D)$ is empty, then we set ${\dim \mathcal{D}(D)=-1}$. In what follows, for a divisor $D$ by $D^{=1}$, we denote the reduced sum of the components of $D$ with coefficient one. For an lc log CY pair $(X, D)$, we define $\mathcal{D}(X, D)$ as $\mathcal{D}(D_Y^{=1})$ where $f\colon (Y, D_Y)\to (X, D)$ is a log resolution of $(X, D)$, so that the formula
\[
K_{Y} + D_Y= f^*(K_X + D)
\]
is satisfied. It is known that the PL-homeomorphism class of $\mathcal{D}(D_Y^{=1})$ does not depend on the choice of a log resolution, see \cite[Proposition 11]{dFKX17}. 
For more results on the topology of dual complexes of Calabi--Yau pairs, see \cite{KX16}. 

\subsection{Coregularity}

We recall the notion of regularity which was introduced in \cite[7.9]{Sh00} and further studied in \cite{Mo24}, \cite{ALP24}.
\begin{definition}
\label{defin-regularity}
Let $X$ be a normal projective variety of dimension $n$ such that the canonical divisor~$K_X$ is $\mathbb{Q}$-Cartier. 
By the $\mathcal{D}(X, D)$ we mean the dual complex of an lc pair $(X, D)$. 
For $l\geqslant 1$, we~define the $l$-th \emph{regularity} of~$X$ by~the formula
\begin{equation}
\label{eq-reg-defin}
\mathrm{reg}_{l}(X) = \max \{ \dim \mathcal{D}(X, D)\ |\ D\in \frac{1}{l} |-lK_X|\}.
\end{equation}
In this situation, $D$ is  an $l$-complement of $K_X$.
Then the \emph{regularity} of $X$ is
\[
\mathrm{reg}(X) = \max_{l\geqslant 1} \{\mathrm{reg}_l(X)\}.
\]
Note that $\mathrm{reg}(X)\in \{-1, 0,\ldots, n-1\}$ where by convention we say that the dimension of~the~empty set is $-1$. The \emph{coregularity} of $X$ is defined as the number
$$
\mathrm{coreg}(X) = n-1-\mathrm{reg}(X).
$$
For convenience, for any $l\geqslant 1$ we also define $\mathrm{coreg}_l(X)=n-1-\mathrm{reg}_l(X)$, so that
$$
\mathrm{coreg}(X)=\min_{l\geqslant 1} \{\mathrm{coreg}_l(X)\}.
$$
\end{definition}

\begin{remark}
\label{rem-non-closed-field}
If $X$ is defined over a field $\mathbb{K}$ of characteristic $0$ which is not algebraically closed, to define $\mathrm{coreg}(X)$ one can use the definition of $G$-coregularity similar to \cite[Definition 2.1.1, Remark 2.1.4]{LPT25} where $G$ is the Galois group of $\overline{\mathbb{K}}$ over $\mathbb{K}$. 
Equivalently, in \eqref{eq-reg-defin} one can consider only the elements of the linear systems $|-lK_X|$ which are defined over $\mathbb{K}$. 
This situation will be considered in Section \ref{sec-dp-over-k}. 
If the ground field $\mathbb{K}$ has positive characteristic, it is not clear whether the dual complex is independent of the representative of the crepant equivalence class of a pair. Assuming that the answer is positive, we formulate  Question \ref{question-positive-char}, see also \cite[4.10]{Mo24} and \cite[31]{dFKX17}. 
\end{remark}

\begin{lemma}
\label{lem-coregpush}
Let $X$ and $Y$ be normal projective varieties, and $f\colon Y \to X$ be a birational morphism. Then for any $l\geqslant 1$ one has $\coreg_l(Y) \geqslant \coreg_l(X)$, and $\coreg(Y) \geqslant \coreg(X)$. In particular, if~$\coreg(Y) = 0$ then $\coreg(X) = 0$.
\end{lemma}

\begin{proof}
Let $l$ be a positive integer, and let $(Y,D)$ be a Calabi--Yau pair on which $\coreg_l(Y)$ is~achieved. Then $(X, f(D))$ is a Calabi--Yau pair, and dual complexes $\mathcal{D}(Y, D)$ and $\mathcal{D}(X, f(D))$ are homeomorphic. Therefore ${\coreg_l(Y) \geqslant \coreg_l(X)}$, and $\coreg(Y) \geqslant \coreg(X)$.
\end{proof}

We will repeatedly use the following important result which allows to check when the coregularity is positive.

\begin{theorem}[{\cite[Theorem 4]{FFMP25}}]
\label{thm-1-2-complement}
Let $X$ be a klt Fano variety defined over algebraically closed field of characteristic $0$. Then if $\mathrm{coreg}(X)=0$ then either $\mathrm{coreg}_1(X)=0$ or $\mathrm{coreg}_2(X)=0$.
\end{theorem}

\subsection{Del Pezzo surfaces}
We recall some well-known properties of del Pezzo surface. By a \emph{del Pezzo surface} over a field $\mathbb{K}$ we mean a normal projective surface $X$ over $\mathbb{K}$ with ample anti-canonical Cartier divisor $-K_X$. We say that a del Pezzo surface has du Val singularities, if it has canonical singularities.  
\begin{theorem}[{cf. \cite{HW81}}]
\label{thm-del_pezzo_properties}
Let $X$ be a del Pezzo surface with du Val singularities over a perfect field $\mathbb{K}$. 
Put $d=(-K_X)^2$. 
Let $\Phi\colon X\dashrightarrow \mathbb{P}^d$ be a rational map given by the linear system $|-K_X|$. 
The following assertions hold:
\begin{enumerate}
    \item If $d \geqslant 3$, then $-K_X$ is very ample, so that $\Phi$ is an embedding;
    \item If $d \geqslant 2$, then $|-K_X|$ is base point free, so that $\Phi$ is a morphism;
    \item If $d = 3$, then $\Phi(X)$ is a cubic surface in $\mathbb{P}^3$;
    \item If $d = 2$, then $X$ is a hypersurface in $\mathbb{P}(1,1,1,2)$ of degree $4$;
    \item If $d = 1$, then $|-K_X|$ is an elliptic pencil, its base locus consists of one point $O \notin \operatorname{Sing}(X)$, and every curve in $|-K_X|$ is irreducible and smooth at $O$;
    \item If $d = 1$
    then $X$ is a hypersurface in $\mathbb{P}(1,1,2,3)$ of degree $6$.
\end{enumerate}
\end{theorem}
\begin{remark}
If in Theorem \ref{thm-del_pezzo_properties} the field $\mathbb{K}$ has characteristic $0$, then if $d=2$ we have that $\Phi$ is a~double cover of $\mathbb{P}^2$ that is branched over a possibly reducible quartic curve; $d = 1$, then $|-2K_X|$ defines a double cover $X \to \mathbb{P}(1,1,2)\subset \mathbb{P}^3$ branched over a sextic curve. 
\end{remark}

Recall that a \emph{weak del Pezzo surface} is a smooth surface such that $-K_Z$ is nef and big.

\begin{lemma}[{cf. \cite[Proposition 8.1.23]{Dol12}}]
\label{lem-blowup}
Let $X$ be a weak del Pezzo surface, and $f\colon Y \to X$ be the blow-up of $X$ at a point $P$ defined over a perfect field $\mathbb{K}$. If $P$ does not lie on a $(-2)$-curve and~${K_X^2 \geqslant 2}$ then $Y$ is a weak del Pezzo surface.
\end{lemma}

For a singular del Pezzo surface $X$ with du Val singularities over a perfect field $\mathbb{K}$ one can consider the minimal resolution of singularities $\pi \colon Z \to X$, where $Z$ is a weak del Pezzo surface. The following lemma allows to study coregularity of weak del Pezzo surfaces instead of the corresponding singular del Pezzo surfaces, and vice versa. 

\begin{lemma}
\label{lem-singweak}
Let $X$ be a singular del Pezzo surface $X$ with du Val singularities over a field $\mathbb{K}$ of~characteristic $0$, and $\pi \colon Z \to X$ be the minimal resolution of singularities. Then for any $l\geqslant 1$ one has $\coreg_l(X) = \coreg_l(Z)$, and $\coreg(X) = \coreg(Z)$.
\end{lemma}

\begin{proof}
For any $l\geqslant 1$ one has $\coreg_l(Z) \geqslant \coreg_l(X)$ by Lemma \ref{lem-coregpush}. Moreover, for any Calabi--Yau pair $(X,D)$ on $X$ the pair $(Z,f^*(D))$ is a Calabi--Yau pair since the singularities of~$X$ are canonical. Thus the dual complexes $\mathcal{D}(X, D)$ and $\mathcal{D}(Z, f^*(D))$ are homeomorphic and $\coreg_l(Z) \leqslant \coreg_l(X)$. So we have $\coreg_l(Z) = \coreg_l(X)$, and $\coreg(X) = \coreg(Z)$.
\end{proof}

The configurations of du Val singularities that can appear on a del Pezzo surface defined over an~algebraically closed field of arbitrary characteristic were described in \cite{St21}.

\section{Del Pezzo surfaces over $\mathbb{C}$}
\label{sec-dp-over-c}
In this section, we are going to compute the coregularity of del Pezzo surfaces with du Val singularities defined an algebraically closed field $\mathbb{K}$ of characteristic $0$. By~Lefschetz principle we mas assume that $\mathbb{K} = \mathbb{C}$.

\begin{lemma}
\label{lem-dp2}
Let $X$ be a du Val del Pezzo surface of degree $d\geqslant 2$. Then $\mathrm{coreg}(X)=\mathrm{coreg}_1(X)=0$.
\end{lemma}
\begin{proof}
We can blow up $d-2$ general points on $X$ to obtain a del Pezzo surface $Y$ of degree $2$ with du~Val singularities together with a morphism $h\colon Y\to X$, cf. Lemma \ref{lem-blowup}. 
Hence, using Lemma \ref{lem-coregpush} we may assume that the degree $(-K_X)^2=2$ from the start. 
By Theorem \ref{thm-del_pezzo_properties}, the linear system~$|-K_X|$ gives a double cover $g\colon X\to \mathbb{P}^2$ ramified along a reduced quartic curve $C\subset \mathbb{P}^2$. Assume that $C$ is not the union of four lines. Then $C$ has a component $C'$ of degree at least $2$. Consider a general tangent line $L$ to $C'$. We claim that $g^*L\sim -K_X$ and $\dim \mathcal{D}(X, g^*L) = 1$. Indeed, this follows from the fact that $g^*L$ is a nodal anti-canonical curve which on $X$ does not pass through singularities of $X$ (because $L$ does not pass through singularities of $C$).  

Consider the case when $C$ is a union of four lines. Note that the case when $C$ is four lines intersecting at one point is not possible. Indeed, in this case we would have a singular point of~the~form 
\[
w^2 + xy(x-y)(x-ay) = 0, \quad \quad \quad a\in \mathbb{C},
\]
which is not du Val. 
We can pick a general line $D$ passing through exactly one nodal point of~$C$. Then $g^*L$ is an irreducible a rational curve with one node and $g^*L\sim -K_X$. 
It follows that $\mathrm{coreg}(X)=\mathrm{coreg}_1(X)=0$.
\end{proof}

Now we deal with the case of del Pezzo surfaces of degree $1$ with du Val singularities. 

\begin{remark}
Note that a general element $C$ in the linear system $|-K_X|$ is a smooth elliptic curve, so $\dim \mathcal{D}(X,C)=0$. This implies that $0\leqslant \mathrm{coreg}_1(X)\leqslant 1$ and $0\leqslant \mathrm{coreg}(X)\leqslant 1$.
\end{remark}

We fix the notation that we will use until the end of this section.

\begin{correspondence}
\label{notation-dp-elliptic-surface}
Let $X$ be a del Pezzo surface with du Val singularities. 
The~linear system $|-K_X|$ has one basepoint $P$ which is a non-singular point on $X$, see Theorem \ref{thm-del_pezzo_properties}. Blowing up this point we obtain a surface $X'$ together with a morphism $f\colon X'\to X$. The linear system~$|-K_{X'}|$ defines a~morphism $\phi\colon X'\to \mathbb{P}^1$ whose general fiber is an elliptic curve. We denote the minimal resolution of singularities of $X'$ by $Y$. 
Put $h\colon Y\to X$. Consider the induced morphism $\phi_Y\colon Y\to \mathbb{P}^1$. Note that $\phi_Y$ is an elliptic fibration with a section, the section being the $f$-exceptional $(-1)$-curve. 
Note that $\phi_Y$ is relatively minimal. 
We summarize the notation in the following diagram (here $\pi\colon Z\to X$ is the minimal resolution of $X$):
\begin{equation}
\label{eq-main-diagram}
\begin{tikzcd}
 Y \ar[d] \ar[dr,"h"] \ar[r] & Z \ar[d, "\pi"] \\
 X' \ar[d, "\phi"] \ar[r,"f"] & X \ar[dl, dashed] \\
 \mathbb{P}^1 &  
\end{tikzcd}
\end{equation}
\end{correspondence}

\begin{lemma}
\label{lem-red-irred-fibers}
Let $X$ be a del Pezzo surface of degree $1$ with at worst du Val singularities. 
The~elements in $|-K_X|$ are reduced and irreducible curves. 
For a given curve $C\in |-K_X|$, at most one singular point of $X$ belongs to $C$. 
\end{lemma}
\begin{proof}
Let $C\in |-K_X|$. The first claim follows either from Theorem \ref{thm-del_pezzo_properties}, or from the equation~${C(-K_X)=(-K_X)^2=1}$. 
We prove the second claim. 
Consider the double cover
$$
g\colon X\to \mathbb{P}(1,1,2)\subset \mathbb{P}^3
$$
given by the linear system $|-2K_X|$. Then $g$ is ramified along a curve $R\subset\mathbb{P}(1,1,2)$ of arithmetic genus $4$ which is the intersection of a quadric cone $\mathbb{P}(1,1,2)$ and a cubic surface in $\mathbb{P}^3$. Note that~$R$ does not pass through the vertex of $\mathbb{P}(1,1,2)$, and the singular points of $X$ correspond to the singular points of $R$. The intersection number of $C$ with a ruling $F$ of the cone $\mathbb{P}(1,1,2)$ is equal to $3$, hence at most one singular point of $C$ belongs to $F$. The map $g$ restricted to any element of~$|-K_X|$ induces a double cover of a ruling $F$ of $\mathbb{P}(1,1,2)$ ramified in $F\cap C$ and the vertex of the cone, and the claim follows. 
\end{proof}

We recall the correspondence between singular points on $X$ and singular fibers of $\phi_Y$ (here $A_0$ denotes a smooth point), see e.g. \cite[Table 3]{St21}. We use the notation of the Kodaira classification of singular fibers of elliptic fibrations, see e.g. \cite[§V.7, V.10]{BHPVdV04}

\

\begin{center}
\begin{tabular}{ | m{5.7em} | m{0.8cm} | m{2cm} | m{0.8cm} | m{0.8cm} | m{0.8cm} | m{0.8cm} | m{2cm} | m{0.8cm} | m{0.8cm} | m{0.8cm} | m{0.8cm}} 
  \hline
  Kodaira type & $I_0$ & $I_n$, ${n\geqslant 1}$ & $II$ & $III$ & $IV$ & $I_0^*$ & $I_n^*$, ${n \geqslant 1}$ & $IV^*$ & $III^*$ & $II^*$\\ 
  \hline
  du Val sing.& $A_0$ & $A_{n-1}$ & $A_0$ & $A_1$ & $A_2$ & $D_4$ & $D_{4+n}$ & $E_6$ & $E_7$ & $E_8$ \\ 
  \hline
    $1-\mathrm{lct}(Y; F)$ & $0\rule[-12pt]{0pt}{32pt}$ & $0\rule[-12pt]{0pt}{32pt}$ & $\dfrac{1}{6}$ & $\dfrac{1}{4}$ & $\dfrac{1}{3}$ & $\dfrac{1}{2}$ & $\dfrac{1}{2}$ & $\dfrac{2}{3}$ & $\dfrac{3}{4}$ & $\dfrac{5}{6}$ \\ 
  \hline
   $\chi(F)$ & $0$ & $n$ & $2$ & $3$ & $4$ & $6$ & $6+n$ & $8$ & $9$ & $10$ \\ 
  \hline
   $j(F)$ & $\in \mathbb{C}$ & $\infty$ & $0$ & $1728$ & $0$ & $\in \mathbb{C}$ & $\infty$ & $0$ & $1728$ & $0$  \\ 
  \hline
\end{tabular}
\end{center}
\begin{center}

\medskip

Table 1. The correspondence between singular fibers of $\phi_Y$ and singular points of $X$
\end{center}

\medskip

\begin{remark}
\label{rem-notbijection}
Note that for a given type of singular fiber of $\phi_Y$ there is a unique type of du Val singularity on $X$. The converse is not true for the singularities $A_0$, $A_1$, $A_2$.
More precisely, a smooth point $A_0$ corresponds to the fibers of type $I_0$, $I_1$, and $II$, a singularity $A_1$ corresponds to the fibers of type $I_2$ and $III$, and a singularity $A_2$ corresponds to the fibers of types $I_3$ and $IV$. For the other types of singularities on $X$ there is a unique corresponding type of fiber on $Y$.
\end{remark}

\begin{lemma}
\label{lem-fiberscoreg}
Let $X$ be a del Pezzo surface of degree $1$ with at worst du Val singularities. 
One has~${\coreg_1 (X) = 0}$ if and only if $\phi_Y$ has a fiber of type $I_n$, $n \geqslant 1$.

Moreover, if $\phi_Y$ has a fiber of type $I_n$ or $I_n^*$, where $n \geqslant 1$, then $\coreg(X) = 0$.
\end{lemma}

\begin{proof}
Note that for any $D \in |-K_X|$ there is a fiber $F$ of $\phi_Y$ such that $D = h(F)$ where $h$ is as in~\eqref{eq-main-diagram}. 
Therefore $\coreg_1 (X)$ is completely determined by the types of fibers of $\phi_Y$, see Remark~\ref{rem-notbijection}. It is easy to see that for a fiber $F$ of $\phi_Y$ one has $\dim \mathcal{D}(Y,F) = \dim \mathcal{D}(X,h(F)) = 1$ if and only if~$F$ has type $I_n$, $n \geqslant 1$. Therefore $\coreg_1 (X) = 0$ if and only if $\phi_Y$ has a fiber of type $I_n$, $n \geqslant 1$.

If $\phi_Y$ has a fiber $F$ of type $I_n^*$, where $n \geqslant 1$, then for a general smooth fiber $F'$ the pair $(Y, \frac{1}{2}F + \frac{1}{2}F')$ is log canonical, and $\dim \mathcal{D}(Y, \frac{1}{2}F + \frac{1}{2}F') = 1$. Thus, $\mathrm{coreg}(X)=0$. 
This proves the second claim.
\end{proof}

\begin{corollary}
\label{cor-a-d-0}
If $X$ has a singular point of type $A_n$ for $n\geqslant 3$ or $D_m$ for $m\geqslant 5$ then $\coreg(X) =~0$. 
\end{corollary}

\begin{proof}
If $X$ has a singular point of type $A_n$ for $n\geqslant 3$ or $D_m$ for $m\geqslant 5$ then $\phi_Y$ has a fiber of~type~$I_{n+1}$ or $I_{m-4}^*$ respectively, see Remark \ref{rem-notbijection}. Therefore $\coreg(X) = 0$ by Lemma~\ref{lem-fiberscoreg}.
\end{proof}

\begin{proposition}
\label{prop-isotrivial}
Let $X$ be a del Pezzo surface of degree $1$ with at worst du Val singularities. 
Then $\coreg(X) > 0$ implies that $\phi_Y$ is isotrivial.

If $\phi_Y\colon Y\to \mathbb{P}^1$ is isotrivial then $\coreg_1(X)=1$. 
\end{proposition}

\begin{proof}
By Lemma \ref{lem-red-irred-fibers} the fibers of $\phi_Y\colon Y\to \mathbb{P}^1$ are not multiple. 
The $j$-invariant function has a pole on $\mathbb{P}^1$ if and only if the corresponding fiber of $\phi_Y$ has type $I_n$ or $I^*_n$ for $n \geqslant 1$, see Table~1.  
Assume that $\phi_Y\colon Y\to \mathbb{P}^1$ is not isotrivial, then $j$ has a pole, and there exists a fiber $F$ with $j(F)=\infty$. Then $\coreg(X) = 0$ by Lemma~\ref{lem-fiberscoreg}.

Assume that 
$\phi_Y$ is isotrivial, so $j \in \mathbb{C}$. In particular, $\phi_Y$ does not have a fiber of type $I_n$, $n \geqslant 1$, and $\coreg_1 (X) = 1$ by Lemma~\ref{lem-fiberscoreg}.
\end{proof}

\begin{remark}
In view of Table 1, Proposition \ref{prop-smooth-dP} can be stated as follows: a smooth del Pezzo surface $X$ of degree $1$ has coregularity $1$ if and only if the corresponding elliptic fibration $\phi_Y\colon Y\to \mathbb{P}^1$ is isotrivial (and has $j$-invariant equal to $0$). Otherwise, $X$ has coregularity $0$.
\end{remark}

\begin{remark}
\label{rem-possible-singular-fibers}
Note that for the topological Euler characteristic $\chi(Y)$ we have $\chi(Y)=12$, and at the same time $\chi(Y)$ is equal to the sum of $\chi(F)$ for singular fibers of $\phi_Y$. Therefore if $\phi_Y$ is isotrivial with a given \mbox{$j$-invariant} then one can apply Table~$1$ and compute possibilities for configurations of~singular fibers of $\phi_Y$. We list the following possibilities, in which, for example, we write $IV + 4II$ if $\phi_Y$ has one singular fiber of Kodaira type $IV$, $4$ singular fibers of Kodaira type $II$, and no other singular fibers.

\begin{itemize}
\item If $j \neq 0$ and $j \neq 1728$ then the only possible configuration is $2I_0^*$.

\item If $j = 0$ then the possible configurations are $II^* + II$, $IV^* + IV$, $IV^* + 2II$, $2I_0^*$, $I_0^* + IV + II$, $I_0^* + 3II$, $3IV$, $2IV + 2II$, $IV + 4II$, $6II$. 

\item If $j = 1728$ then the possible configurations are $III^* + III$, $2I_0^*$, $I_0^*+2III$, $4III$.
\end{itemize}
Note that the list of these configurations is well known and can be found, for example, in \cite[Section 4]{Mi90}. 
\end{remark}

We will use the following well-known result.

\begin{proposition}[{cf. \cite[Example 1.3.1 and 1.3.2]{Sil94}}]
\label{rem-cm-ell-curves}
An elliptic curve with $j=0$ (resp., $j=1728$) is given by the following Weierstrass equation:
\begin{equation}
w^2 = z^3 + a, \quad a\neq 0 \quad \quad \quad (\text{resp},\ \ w^2 = z^3 + b z, \quad b\neq 0). 
\end{equation}
These elliptic curves have complex multiplication and admit an automorphism of order $6$ (resp.,~$4$). They can be obtained as a quotient $\mathbb{C}/\Lambda$ where $\Lambda$ is the lattice of Eisenstein (resp., Gaussian) numbers in $\mathbb{C}$.  
\end{proposition}

Now we want to consider different cases of isotrivial elliptic fibrations $\phi_Y$ and describe the corresponding del Pezzo surfaces $X$ with $\coreg(X) = 1$. Note that the configuration $2I_0^*$ corresponding to~two $D_4$-singularities on $X$ may have any value  of $j$. Actually this case was considered in~\cite[Example 3.13]{Mo24}, and one has $\coreg(X) = 0$ in this case. We give a proof of the fact that~${\coreg(X) = 0}$ for convenience of the reader.

\begin{proposition}
\label{prop-2D4}
Let $X$ be a del Pezzo surface of degree $1$ with two $D_4$-singularities. Then $\mathrm{coreg}_1(X)=1$, $\coreg_2(X)=0$, so $\coreg(X)=0$.
\end{proposition}

\begin{proof}
Two singular fibers of $\phi_Y$ containing $D_4$-singularities have type $I_0^*$, and there are no other singular fibers, since $\chi(I^*_0)=6$, and $2\chi(I^*_0) = \chi(Y)=12$. Therefore $\coreg_1(X) = 1$ by Lemma~\ref{lem-fiberscoreg}. 

To compute $\coreg_2(X)$, consider the equation (cf. Big Table in \cite[Section 8]{ChP08}) of $X$ in~$\mathbb{P}(1,1,2,3)$:
\[
w^2=z(z+xy)(z+axy),\quad \quad a\in \mathbb{C}\setminus \{0,1\}.
\]
We see that the Cartier divisor $D=\{z=0\}$ belongs to the linear system $|-2K_X|$, and passes through both singular points of $X$ which have coordinates $P_1=(1:0:0:0)$ and $P_2=(0:1:0:0)$. 
One has $D=2D'$ for a reduced irreducible curve $D'$ given by $z=w=0$. 
Let us check that the pair~${(X,D')}$ is lc, and there exists a divisor $E$ with log discrepancy $0$ over $P_i$. We can work in the affine chart $\mathbb{C}^3$ of $\mathbb{P}(1,1,2,3)$ near $P_i$ where, up to a coordinate change, we have the following equations of $X$ and $D$:
\[
w^2 = xz(z+x), \quad \quad \quad D = \{z=0\},
\]
and $P_i$ is the origin. 
Consider the blow-up of the origin $\psi\colon \widetilde{\mathbb{C}^3}\to \mathbb{C}^3$. Working in the chart $z\mapsto xz$, $w\mapsto xw$, $x\mapsto x$ on the blow-up, we obtain the following equation of the strict transform~$\widetilde{X}$ of $X$,
\[
w^2 = xz(z+1),
\]
which is the equation of an ordinary double point. The preimage $\widetilde{D}'=\psi|_{\widetilde{X}} ^*D'$ of $D'$ is given by the equation $xz=0$, so $(\widetilde{X}, \widetilde{D}')$ is strictly lc. Hence, $(X,D')$ is lc, and there exists a divisor $E$ with log discrepancy $0$ over $P_i$. Thus, $\mathrm{coreg}(X)=\mathrm{coreg}_2(X)=0$ as claimed. 
\end{proof}

Hence we see that if $\coreg(X) = 1$ then $\phi_Y$ is isotrivial with $j$-invariant $0$ or $1728$. For these cases we need the following proposition.

\begin{proposition}
\label{prop-irreducible-coreg}
Let $X$ be a del Pezzo surface of degree $1$ with du Val singularities. Consider a~double cover~${g\colon X\to \mathbb{P}(1,1,2)}$ with the ramification curve $R\sim 6L$ where $L$ is the positive generator of $\mathrm{Cl}(\mathbb{P}(1,1,2))=\mathbb{Z}$. Assume that $R$ is irreducible. Let $D$ be an element of $|-2K_X|$ which is not of the form $D = D_1 + D_2$ with $D_i\in |-K_X|$. 
Then $(X, \frac{1}{2}D)$ is~klt. 

As a consequence, either $X$ has a singular point of type $D_n$ for $n\geqslant 5$ and $\mathrm{coreg}(X)=0$, or~${\mathrm{coreg}(X)=\mathrm{coreg}_1(X)}$.
\end{proposition}

\begin{proof}
By Theorem \ref{thm-del_pezzo_properties}, the surface $X$ can be given in $\mathbb{P}(1,1,2,3)$ by the equation:
\begin{equation}
\label{eq-general-delpezzo-degree1}
w^2 + z^3 + z f_4(x,y) + f_6(x,y) = 0,
\end{equation}
where $f_i(x,y)$ is a homogeneous polynomial of degree $i$.

The equation of $D$ 
is $z + g_2(x,y)=0$, since $D$ is not of the form $D_1+D_2$ with $D_i\in |-K_X|$. Substituting this equation to \eqref{eq-general-delpezzo-degree1}, we get an equation 
\[
w^2 = g_6(x,y)
\]
in $\mathbb{P}(1,1,3)$. 
Note that $g_6$ is a non-zero polynomial, which follows from the assumption that $R$ is~irreducible, so $g(D)$ is not a component of $R$. 
Note that the singular point of $\mathbb{P}(1,1,3)$ does not belong to $D$. In the charts isomorphic to $\mathbb{C}^2$ which cover $D$ the equation of $D$ becomes (after an~analytic change of coordinates) 
\[
w^2 = x^a, \quad \quad \quad w^2 = y^b,
\]
where $a \geqslant 1$, $b \geqslant 1$. We conclude that the log-canonical threshold of the pair $(X, \frac{1}{2}D)$ is greater than $\frac{1}{2}$. In other words, $(X, \frac{1}{2}D)$ is klt.

If $X$ has a singular point of type $D_n$ for $n\geqslant 5$ then $\mathrm{coreg}(X)=0$ by Corollary \ref{cor-a-d-0}. Otherwise, for the pair $(X, \frac{1}{2}D)$ with $D=D_1+D_2$ and $D_i\in |-K_X|$ one has $\dim \mathcal{D}(X,\frac{1}{2}D)\leqslant 0$, cf. Table 1. Using Theorem \ref{thm-1-2-complement}, this implies that $\mathrm{coreg}(X)=\mathrm{coreg}_1(X)$, and the claim follows. 
\end{proof}

\subsection{Isotrivial fibrations with $j=0$}
\label{sec-j-0}
Now we show that for isotrivial $\phi_Y$ with $j = 0$ the assumptions of Proposition \ref{prop-irreducible-coreg} are satisfied almost always.

\begin{lemma}
\label{lem-irreducible-j0}
Let $X$ be a del Pezzo surface of degree $1$ with du Val singularities such that the~corresponding elliptic fibration $\phi_Y$ is isotrivial with $j$-invariant $0$. Then either the ramification curve~$R$ of the double cover $g\colon X\to \mathbb{P}(1,1,2)$ is irreducible, or $X$ has two $D_4$-singularities.
\end{lemma}

\begin{proof}
If $\phi_Y$ is isotrivial and $j = 0$ then equation~\eqref{eq-general-delpezzo-degree1} in $\mathbb{P}(1,1,2,3)$ has the form (cf. Proposition~\ref{rem-cm-ell-curves})
$$
w^2 + z^3 + f_6(x,y) = 0.
$$
The ramification divisor $R$ in $\mathbb{P}(1,1,2)$ is given by the equation $z^3 + f_6(x,y) = 0$. Note that the~curve~$R$ is irreducible if and only if the polynomial $z^3 + f_6(x,y)$ is irreducible.

Assume that $z^3 + f_6(x,y)$ is reducible. Then it has an irreducible factor $z + p(x,y)$, and~therefore $(-p(x,y))^3 + f_6(x,y) = 0$. Thus $f_6(x,y)$ is a cube of $p(x,y)$, and $p(x,y)$ is a homogeneous polynomial of degree $2$. After a coordinate change, there are two possibilities: either~${p(x,y) = x^2}$, or~${p(x,y) = xy}$. The former case is impossible, since the surface $X$ given by the equation
$$
w^2 + z^3 + x^6 = 0
$$
has a singular point $w=z=x=0$ which is not a du Val singularity. In the latter case $X$ is given by~the~equation~$w^2 + z^3 + x^3y^3 = 0$, and $X$ is a del Pezzo surface of degree $1$ with two $D_4$-singularities, which can be seen by passing to affine charts, see the proof of Proposition \ref{prop-2D4}.
\end{proof}

Now we are able to obtain the following proposition in which we compute $\coreg(X)$ for the case when $\phi_Y$ is isotrivial with $j = 0$.

\begin{proposition}
\label{prop-j0-coreg-1-except-for-d4}
Let $X$ be a del Pezzo surface of degree $1$ with du Val singularities such that the corresponding elliptic fibration $\phi_Y$ is isotrivial with $j$-invariant $0$. If $X$ has two $D_4$-singulartities then $\coreg(X) = 0$. Otherwise $\coreg(X) = 1$.
\end{proposition}

\begin{proof}
If $X$ has two $D_4$-singularities then $\coreg(X) = 0$ by Proposition~\ref{prop-2D4}. Otherwise the ramification curve $R$ of the double cover $g\colon X\to \mathbb{P}(1,1,2)$ is irreducible by Lemma~\ref{lem-irreducible-j0}, and therefore~${\coreg(X) = \coreg_1(X) = 1}$ by Propositions~\ref{prop-irreducible-coreg} and \ref{prop-isotrivial}. 
\end{proof}

\begin{remark}
\label{rem-equations-j-0}
One can easily find $f_6(x,y)$ in the equation
$$
w^2 + z^3 + f_6(x,y) =0
$$
giving a del Pezzo surface $X$ of degree $1$ with du Val singularities such that the corresponding elliptic fibration $\phi_Y$ is isotrivial with $j$-invariant $0$ for different configurations of singular fibers of~$\phi_Y$. This is because the types of possible singular fibers with \mbox{$j$-invariant} $0$ are determined by the multiplicity of roots of $f_6(x,y)$. We collect this data in the following table, in which $\rho(X)$ denotes the Picard number of $X$ and $\dim M$ denotes the dimension of the moduli space of such surfaces.

\medskip

\begin{center}
\begin{tabular}{|c|c|c|c|c|c|} 
  \hline
  Fibers of $\phi_Y$ & $\mathrm{Sing}(X)$ & $f_6(x,y)$ & $\coreg(X)$ & $\rho(X)$ & $\dim M$ \\
  \hline
  $II^* + II$ & $E_8$ & $x^5y$ & $1$ & $1$ & $0$ \\ 
  \hline
  $IV^* + IV$ & $E_6 + A_2$ & $x^4y^2$ & $1$ & $1$ & $0$ \\ 
  \hline
  $IV^* + 2II$ & $E_6$ & $x^4y(x-y)$ & $1$ & $3$ & $0$ \\ 
  \hline
  $2I_0^*$ & $2D_4$ & $x^3y^3$ & $0$ & $1$ & $0$ \\ 
  \hline
  $I_0^* + IV + II$ & $D_4 + A_2$ & $x^3y^2(x-y)$ & $1$ & $3$ & $0$  \\ 
  \hline
  $I_0^* + 3II$ & $D_4$ & $x^3y(x-y)(x-ly)$ & $1$ & $5$ & $1$ \\ 
  \hline
  $3IV$ & $3A_2$ & $x^2y^2(x-y)^2$ & $1$ & $3$ & $0$ \\ 
  \hline
  $2IV + 2II$ & $2A_2$ & $x^2y^2(x-y)(x-ly)$ & $1$ & $5$ & $1$\\ 
  \hline
  $IV + 4II$ & $A_2$ & $x^2y(x-y)(x-ly)(x-my)$ & $1$ & $7$ & $2$ \\ 
  \hline
  $6II$ & $\varnothing$ & $xy(x-y)(x-ly)(x-my)(x-ny)$ & $1$ & $9$ & $3$ \\ 
  \hline
\end{tabular}
\end{center}
\begin{center}
Table 2. Singular fibers of isotrivial elliptic fibrations with $j=0$
\end{center}

\medskip

Note that a surface with a given configuration of singular fibers is unique if~${\rho(X) \leqslant 3}$, and for~${\rho(X) \geqslant 5}$ one has 
$\dim M = \frac{1}{2}(\rho(X) - 3)$.
\end{remark}

\begin{remark}
\label{rem-equations-j-0-bad}
The configuration of singularities does not define the surface $X$ (and the elliptic fibration $\phi_Y$) uniquely since in general case $\phi_Y$ may not be isotrivial, or can be isotrivial with $j \neq 0$. For example, a one-dimensional family of surfaces with configuration $2D_4$ and different $j$-invariants is constructed in Proposition~\ref{prop-2D4}.

A surface with a unique $E_8$-singularity can also give collection of singular fibers $II^* + 2I_1$ for $\phi_Y$. In this case $\phi_Y$ is not isotrivial and has two nodal fibers, therefore~${\coreg(X) = 0}$. One can check that such $X$ is unique, and can be given by the equation in $\mathbb{P}(1,1,2,3)$ (cf. \cite[Remark B.5]{ChP08})
\[
w^2 + z^3 + x^4z + x^5y = 0.
\]

A surface with $E_6$ and $A_2$ singularities can also give collection of singular fibers ${IV^* + I_3 + I_1}$ for~$\phi_Y$. In this case $\phi_Y$ is not isotrivial and has a nodal fiber, therefore $\coreg(X) = 0$. 
The existence of such a surface $X$ is proven in \cite{Ye02}, see Table 4.1.
In fact, one can write down an~explicit equation of $X$ such that a surface with singular fibers ${IV^* + IV}$ for~$\phi_Y$ is its specialization:
$$
w^2+z^3-3x^3(x+4y)z+2x^4(x^2+6xy+6y^2)=0.
$$

\end{remark}

\subsection{Isotrivial fibrations with $j=1728$}
\label{sec-j-1728}
We consider the remaining case $j = 1728$. If $X$ is a del Pezzo surface of degree $1$ such that $\phi_Y$ is isotrivial with $j$-invariant $1728$ then the equation~\eqref{eq-general-delpezzo-degree1} in~$\mathbb{P}(1,1,2,3)$ has the form (cf. Proposition \ref{rem-cm-ell-curves})
$$
w^2 + z^3 + f_4(x,y)z = 0,
$$
where $f_4(x,y)$ is a homogeneous polynomial of degree $4$.

The ramification curve $R$ of the double cover $g\colon X\to \mathbb{P}(1,1,2)$ is reducible in this case, therefore we cannot apply Proposition~\ref{prop-irreducible-coreg}. So according to Remark \ref{rem-possible-singular-fibers} we have to consider four possible configurations of singular fibers of $\phi_Y$. For these configuration we need to find explicit equations of~$X$. We start with the following remark.

\begin{remark}
\label{rem-equations-j-1728}
We find $f_4(x,y)$ in the equation
$$
w^2 + z^3 + f_4(x,y)z =0
$$
giving a del Pezzo surface $X$ of degree $1$ with du Val singularities such that the corresponding elliptic fibration $\phi_Y$ is isotrivial with $j$-invariant $1728$ for different configurations of singular fibers of $\phi_Y$. To do this, note that the types of possible singular fibers with \mbox{$j$-invariant} $1728$ bijectively correspond to multiplicity of roots of $f_4(x,y)$. We collect this data in the following table.

\medskip

\begin{center}
\begin{tabular}{|c|c|c|c|c|c|} 
  \hline
  Fibers of $\phi_Y$ & $\mathrm{Sing}(X)$ & $f_4(x,y)$ & $\coreg(X)$ & $\rho(X)$ & $\dim M$\\
  \hline
  $III^* + III$ & $E_7 + A_1$ & $x^3y$ & $1$ & $1$ & $0$ \\ 
  \hline
  $2I_0^*$ & $2D_4$ & $x^2y^2$ & $0$ & $1$ & $0$ \\ 
  \hline
  $I_0^* + 2III$ & $D_4 + 2A_1$ & $x^2y(x-y)$ & $0$ & $3$ & $0 $\\ 
  \hline
  $4III$ & $4A_1$ & $xy(x-y)(x-ly)$ & $1$ & $5$ & $1$ \\ 
  \hline
\end{tabular}
\end{center}
\begin{center}
Table 3. Singular fibers of isotrivial elliptic fibrations with $j=1728$
\end{center}

\medskip

Note also that a surface with a given configuration of singular fibers is unique if~${\rho(X) \leqslant 3}$, and for $\rho(X) \geqslant 5$ one has $\dim M = \frac{1}{2}(\rho(X) - 3)$.
\end{remark}

\begin{remark}
\label{rem-equations-j-1728-bad}
The configuration of singularities does not uniquely define the surface $X$ (and the elliptic fibration $\phi_Y$) since in general case $\phi_Y$ can be not isotrivial, or can be isotrivial with $j \neq 1728$. For example, one-dimensional family of surfaces with configuration $2D_4$ and different $j$-invariants is constructed in Proposition~\ref{prop-2D4}.

A surface with $E_7$ and $A_1$ singularities can also give collection of singular fibers ${III^* + I_2 + I_1}$ for~$\phi_Y$. In this case $\phi_Y$ is not isotrivial and has a nodal fiber, therefore $\coreg(X) = 0$. The~existence of such a surface $X$ is proven in \cite{Ye02}, see Table 4.1.
In fact, one can write down an~explicit equation of $X$ such that a surface with singular fibers ${III^* + III}$ for~$\phi_Y$ is its specialization:
$$
w^2+z^3-3x^3(x+2y)z+2x^5(x+3y)=0.
$$

\end{remark}

For the case $2I_0^*$ the coregularity is computed in Proposition~\ref{prop-2D4}, and the other cases are considered in Lemmas~\ref{lem-x1e7'}, \ref{lem-x1d42a1} and \ref{lem-x14a1}. To compute $\coreg(X)$ for these cases we need the following lemma.

\begin{lemma}
\label{lem-local-coreg}
Let $x\in X$ be a germ of du Val singularity. If $x\in X$ has type $D_n$, then $x\in X$ does not admit non-trivial $1$-complement (here non-trivial means that $D\not\sim 0$). If $x\in X$ has type $E_n$, then $x\in X$ does not admit a complement of coregularity~$0$.   
\end{lemma}

\begin{proof}
The first claim follows from \cite[Proposition 3.11]{Mo24}. 
The second claim follows from the fact that type $E$ singularity is exceptional, which means that for any divisor $D$ the pair $(X, D)$ admits at~most one exceptional divisor over $x\in X$ with log-discrepancy $0$. Indeed, if $(X, D)$ is a~complement of coregularity $0$, in the log-resolution $(Y, D_Y)$ of $(X, D)$ there will be two intersecting curves $E_1$ and $E_2$ in $D_Y^{=1}$. Blowing-up the point of  intersection of $E_1$ and $E_2$ we would obtain another curve over $x\in X$ with log-discrepancy $0$. This contradiction completes the proof. 
\end{proof}

\begin{lemma}
\label{lem-x1e7'}
Let $X$ be a del Pezzo surface $X$ of degree $1$ with du Val singularities such that the~configuration of singular fibers of the corresponding elliptic fibration $\phi_Y$ is $III+III^*$. 
In~particular, $\mathrm{Sing}(X)=E_7+A_1$. 
Then $\coreg(X) = 1$.
\end{lemma}
\begin{proof}
According to Remark \ref{rem-equations-j-1728}, the surface $X$ is given by the following  equation in $\mathbb{P}(1,1,2,3)$:
$$
w^2 + z(z^2 + x^3y) = 0.
$$
We have $\mathrm{coreg}_1(X)=1$ by Proposition~\ref{prop-isotrivial}. We show that $\mathrm{coreg}_2(X)=1$. 
The singular points of~$X$ have coordinates $P_1=(0:1:0:0)$ and $P_2=(1:0:0:0)$. In the affine charts $\mathbb{C}^3$ of $\mathbb{P}(1,1,2,3)$ that contain $P_1$ (resp., $P_2$), the equation of $X$ becomes
$$
w^2 + z^3 + x^3z = 0
$$
and
$$
w^2 + z^3 + yz=0,
$$
respectively. 
Consider the divisor $D$ given by the equation $z=0$ which passes through both points~$P_1$ and~$P_2$. By Lemma \ref{lem-local-coreg}, we see that the pair $(X, \frac{1}{2}D)$ cannot have coregularity $0$ near~$P_1$. By~a~straightforward computation, we see that $(X, \frac{1}{2}D)$ does not have coregularity $0$ near~$P_2$. 

As for the pairs $(X,\frac{1}{2}B)$ where $B\neq D$, $B\in |-2K_X|$, and $B=\{z+f_2(x,y)=0\}$, arguing as in the proof of Proposition \ref{prop-irreducible-coreg} we have that $(X,\frac{1}{2}B)$ is klt. It follows that $\mathrm{coreg}_2(X)=1$, and so~${\mathrm{coreg}(X)=1}$.
\end{proof}

\begin{lemma}
\label{lem-x1d42a1}
Let $X$ be a del Pezzo surface $X$ of degree $1$ with du Val singularities such that the~configuration of singular fibers of the corresponding elliptic fibration $\phi_Y$ is $I_0^*+2III$. 
In~particular, $\mathrm{Sing}(X)=D_4+2A_1$. 
Then $\coreg(X) = 0$.
\end{lemma}
\begin{proof}
According to Remark \ref{rem-equations-j-1728}, the surface $X$ is given by the following  equation in $\mathbb{P}(1,1,2,3)$:
$$
w^2 + z(z^2 + x^2y(x-y)) = 0
$$
We claim that $\mathrm{coreg}(X)=0$ in this case. Note that the ramification curve $R$ is given in $\mathbb{P}(1,1,2)$ by the equation $z(z^2 + x^2y(x-y)) = 0$. Thus $R$ consists of two irreducible components $C_1\sim H$ and~${C_2\sim 2H}$ where $H$ is a hyperplane section of $\mathbb{P}(1,1,2)\subset \mathbb{P}^3$. Then $C_1$ is a smooth curve, and~$C_2$ has a node at the point $(0:1:0)$ that is the image of $D_4$-singularity. Also, $C_1\cdot C_2=2H^2=4$, so~$C_1$ intersects $C_2$ in two more points $P_1$ and $P_2$ transversally.
Considering the pair $(X,\frac{1}{2}g^*(C_1))$, we conclude that $\mathrm{coreg}(X)=0$.
\end{proof}

\begin{lemma}
\label{lem-x14a1}
Let $X$ be a del Pezzo surface $X$ of degree $1$ with du Val singularities such that the configuration of singular fibers of the corresponding elliptic fibration $\phi_Y$ is $4III$. 
In particular, $\mathrm{Sing}(X)=4A_1$. 
Then $\coreg(X) = 1$.
\end{lemma}

\begin{proof}
According to Remark \ref{rem-equations-j-1728}, the surface $X$ is given by the following  equation in $\mathbb{P}(1,1,2,3)$:
$$
w^2 + z(z^2 + xy(x-y)(x-ly)) = 0
$$
We claim that $\mathrm{coreg}(X)=1$. We have $\mathrm{coreg}_1(X)=1$ by Proposition~\ref{prop-isotrivial}. The ramification curve $R$ is reducible. Therefore we only need to consider the divisor $D$ given by the equation~${z+f_2(x,y)=0}$. In this case, by a local computation involving blowing-up of singular points of $X$ we obtain $\mathrm{coreg}_2(X)=1$. We conclude that $\mathrm{coreg}(X)=1$.
\end{proof}

Now we are ready to prove the main theorem of this section. 

\begin{proof}[{Proof of Theorem \ref{thm-dp1-main}}]
By Proposition \ref{prop-smooth-dP}, we may assume that $X$ is not smooth. By Lemma~\ref{lem-dp2}, if~${(-K_X)^2\geqslant 2}$ then $\mathrm{coreg}_1(X)=0$. So we may assume that $(-K_X)^2 = 1$. 
We shall use the~canonical correspondence~\ref{notation-dp-elliptic-surface}. In particular, by $\phi_Y\colon Y\to \mathbb{P}^1$ we denote the corresponding elliptic fibration. If $\phi_Y$ is not isotrivial, then $\coreg(X) =0$ by Proposition \ref{prop-isotrivial}. Hence we may assume that $\phi_Y$ is~isotrivial, and the $j$-function on the base $\mathbb{P}^1$ is constant. In this case $\coreg_1(X) = 1$. The possible configurations of singular fibers of $\phi_Y$ are listed in Remark \ref{rem-possible-singular-fibers}. 

If $j\notin \{0, 1728\}$, then the configuration of singular fibers of $\phi_Y$ is $2I_0^*$. In this case, $\mathrm{coreg}(X)=0$ by Proposition \ref{prop-2D4}. Assume that $j=0$. We deal with this case in Section \ref{sec-j-0}. 
In Proposition \ref{prop-j0-coreg-1-except-for-d4} we give a criterion when $\mathrm{coreg}(X)=1$ for such surfaces. 
Possible equations of the surfaces are given in Remark \ref{rem-equations-j-0}.
Assume that $j=1728$. We deal with this case in Section \ref{sec-j-1728}. Possible equations of the surfaces are given in Remark \ref{rem-equations-j-1728}. We consider all the possible cases in Proposition~\ref{prop-2D4} and  Lemmas~\ref{lem-x1e7'}, \ref{lem-x1d42a1} and \ref{lem-x14a1}. 

We conclude that  $\mathrm{coreg}(X)=0$ if and only either
$\phi_Y$ is not isotrivial, or 
 $\phi_Y$ is isotrivial and~${\mathrm{Sing}(X)\in \{2D_4, D_4+2A_1\}}$.
Otherwise, $\mathrm{coreg}(X)=1$. 
This finishes the proof.
\end{proof} 

Proof of Theorem \ref{main-thm2} (resp., Theorem \ref{cor-main}) follows from the proof of Theorem \ref{thm-dp1-main} keeping in~mind lists of configurations of singular points on $X$ (resp., singular fibers of $\phi_Y$) given in Table 2 and~Table~3. 
By Remarks~\ref{rem-equations-j-0-bad} and~\ref{rem-equations-j-1728-bad}, there exist examples of du Val del Pezzo surfaces of~degree~$1$ and~Picard rank $1$ with coregularity $0$ and with the given configurations of singular points.


\begin{proof}[{Proof of Theorem \ref{cor-toric-models}}] 
By Lemma \ref{lem-dp2}, if $(-K_X)^2\geqslant 2$, then $\mathrm{coreg}_1(X)=0$, so $X$ admits a toric model by \cite[Lemma 1.13]{GHK15}. Hence we may assume that $(-K_X)^2=1$. 
Assume that $\phi_Y$ is not isotrivial. 
Then $j$-invariant function has a pole on $\mathbb{P}^1$, and therefore $\phi_Y$ has a fiber of type $I_n$ or~$I_n^*$, where $n \geqslant 1$, see Table~$1$. If the first possibility is realized, then $\mathrm{coreg}_1(X)=0$ by Lemma \ref{lem-fiberscoreg}, which is equivalent to the existence of a toric model by \cite[Lemma 1.13]{GHK15}. 

Assume that $\phi_Y$ has a fiber of type $I_n^*$, where $n \geqslant 1$, and no fibers of type $I_n$ for $n\geqslant 1$. According to Table 1, this means that $X$ has a $D_m$-singularity for $m\geqslant 5$. 
Using Table 1 and the~fact that~$\chi(Y)=12$, we obtain that only the following configurations of singular fibers of $\phi_Y$ are possible:
\begin{equation}
\label{eq-degenerate-fibers-I^*_n}
I_1^*+III+II, \quad \quad 
I_2^*+2II, \quad \quad 
I_2^*+IV, \quad \quad 
I_3^*+III, \quad \quad 
I_4^*+II.
\end{equation}
By \cite{Mi90}, the last three cases in \eqref{eq-degenerate-fibers-I^*_n} do not occur. In the first two cases, the degree of the $j$-invariant function is equal to $1$ and $2$, respectively. 
In particular, the elliptic fibration $\phi_Y$ is not isotrivial for such surfaces. 
Such surfaces exists as shown in Example \ref{ex-very-special-surfaces}. If $X$ is one of the two surfaces from Example \ref{ex-very-special-surfaces}, then $\mathrm{coreg}_1(X)=1$, and  $\mathrm{coreg}_2(X)=\mathrm{coreg}(X)=0$ (the last equality follows from Theorem \ref{cor-main} combined with Theorem \ref{thm-1-2-complement}). 

Conversely, if $\phi_Y$ is isotrivial, then $\phi_Y$ has no fibers of type $I_n$ or $I_n^*$, where $n \geqslant 1$. By~Lemma~\ref{lem-fiberscoreg} this implies that $\mathrm{coreg}_1(X)=1$. The proof is completed. 
\end{proof}

\begin{example}
\label{ex-very-special-surfaces}
We denote by $X'_1(D_5+A_1)$ (resp., $X'_1(D_6)$) du Val del Pezzo surface whose configuration of singular points is $D_5+A_1$ (resp., $D_6$) with the additional property that the associated elliptic fibration $\phi_Y$ has the configuration of singular fibers $I_1^*+III+II$ (resp., $I_2^*+2II$).

An example of a surface $X'_1(D_5+A_1)$ is given by the following equation in $\mathbb{P}(1,1,2,3)$:
\[
w^2 + z^3 - 3(x-y)xy^2z + 2(x-y)x^2y^3 = 0.
\]
Let us show that for the surface $X$ given by this equation the associated elliptic fibration $\phi_Y$ has the configuration of singular fibers $I_1^*+III+II$. The discriminant of $z^3 - 3(x-y)xy^2z + 2(x-y)x^2y^3$ is equal to
$$
4 \cdot 27(x-y)^3x^3y^6 - 27 \cdot 4(x-y)^2x^4y^6 = -108(x-y)^2x^3y^7.
$$
Therefore $\phi_Y$ has three singular fibres with $\chi(F)$ equal to $2$, $3$ and $7$, respectively, and the $j$-invariant function has degree $1$. This is possible only for the configuration of singular fibers $I_1^*+III+II$ (see Table $1$).

An example of a surface $X'_1(D_6)$ is given by the following equation in $\mathbb{P}(1,1,2,3)$:
\[
 w^2 + z^3 - 3(x^2-y^2)y^2z + 2(x^2-y^2)xy^3 = 0.
\]
Let us show that for the surface $X$ given by this equation the associated elliptic fibration $\phi_Y$ has the configuration of singular fibers $I_2^*+2II$. The discriminant of $z^3 - 3(x^2-y^2)y^2z + 2(x^2-y^2)xy^3$ is equal to
$$
4 \cdot 27(x^2-y^2)^3y^6 - 27 \cdot 4(x^2-y^2)^2x^2y^6 = -108(x^2-y^2)^2y^8.
$$
Therefore $\phi_Y$ has three singular fibres with $\chi(F)$ equal to $2$, $2$ and $8$ respectively, and the $j$-invariant function has degree $2$. This is possible only for the configurations of singular fibers $IV^*+II+I_2$ and $I_2^*+2II$ (see Table $1$). Note that $X$ has an automorphism $(x:y:z:t)\mapsto(ix:-iy:z:t)$, which permutes the singular fibers given by $x = y$ and $x = -y$. Thus these fibers have the same type, the configuration $IV^*+II+I_2$ is impossible, and the configuration $I_2^*+2II$ is achieved on~$X$.

\end{example}
\begin{remark}
Denote by $X_1(2D_4)$ a du Val del Pezzo surfaces of degree $1$ with two $D_4$-singularities, and by $X'_1(E_8)$, $X'_1(E_7+A_1)$, $X'_1(E_6+A_2)$ the surfaces with corresponding configuration of~singularities and no nodal elements in $|-K_X|$. 
As proven in \cite{ChP08} (see also \cite{Vir24}), the surfaces~$X'_1(E_8)$, $X'_1(E_7+A_1)$, $X'_1(E_6+A_2)$ and $X_1(2D_4)$ have infinite automorphism group. In fact, they are the~only du Val del Pezzo surfaces of degree $1$ with infinite automorphism group.
\end{remark}

\begin{remark}
\label{rem-EFM24}In \cite[Theorem 3.3]{EFM24}, it is proven that any del Pezzo surface of Picard rank~$1$, except for 
\begin{equation}
\label{eq-moraga-exceptions}
X'_1(E_8),\quad\quad X'_1(E_7+A_1),\quad\quad X'_1(E_6+A_2),\quad\quad X_1(2D_4)
\end{equation}
are crepant equivalent to a pair with coregularity $0$, and these exceptional surfaces are not crepant equivalent to a toric pair. In other words, the authors computed $\coreg_1(X)$ for du Val del Pezzo surfaces of Picard rank $1$. One can check that the surfaces \eqref{eq-moraga-exceptions} are the only du Val del Pezzo surfaces of Picard rank $1$ such that the corresponding elliptic fibration $\phi_Y$ is isotrivial. Thus, the result of \cite[Theorem 3.3]{EFM24} is consistent with Theorem \ref{cor-toric-models}.  Also note that the surfaces \eqref{eq-moraga-exceptions} are the only du Val del Pezzo surfaces of Picard rank $1$ which are not uniquely determined by their configuration of singularities, see \cite{Ye02}.  


\end{remark}

\section{Del Pezzo surfaces over non-closed fields}
\label{sec-dp-over-k}
In this section, we treat del Pezzo surfaces defined over a field $\mathbb{K}$ of characteristic $0$ which is not necessarily algebraically closed. Note that in \cite{LPT25} the authors considered a geometric analog of this case, that is, the case of del Pezzo surfaces defined over $\mathbb{C}$ endowed with an action of an~automorphism group $G$. 

We start from the following proposition.

\begin{proposition}
\label{prop-dp6789}
Let $\mathbb{K}$ be a field of characteristic $0$, and let $X$ be a del Pezzo surface of degree $6$ or greater with at worst du Val singularities defined over $\mathbb{K}$. Then $\coreg(X) = \mathrm{coreg}_1(X)=0$.
\end{proposition}

\begin{proof}
Let  $d = (-K_X)^2$ be a degree of del Pezzo surface. We consider possible values of $d$ from $6$ to $9$. 

Assume that $d = 6$. Then there are six possibilities for the configuration of lines and singular points on $X$ (see \cite[Table 8.4]{Dol12}). Let $Z \to X$ be the minimal resolution of singularities.

At first, consider the cases where $X$ is a $\mathbb{K}$-form of a toric surface (these cases are labeled by (i), (ii), (ii') and (iii') in \cite[Table 8.4]{Dol12}). If $X$ has type (i) then $X$ is non-singular and the union~$D$ of the six $(-1)$-curves is a hexagon of transversally intersecting lines. Moreover, $D \in |-K_X|$. Therefore $(X,D)$ is a log canonical pair, and $\coreg_1(X) = 0$.

For the case (ii) consider a union $E$ of $(-1)$- and $(-2)$-curves on $Z$. Then one has $(-K_Z-E)^2 = 0$, and the linear system $|-K_Z - E|$ gives a conic bundle structure on $Z$. For a general $F \in |-K_Z - E|$ the pair $E+F$ is log canonical. Therefore by Lemma~\ref{lem-singweak} one has $\coreg_1(X) = \coreg_1(Z) = 0$.

For the cases (ii') and (iii') the configuration of negative curves on $Z$ is not symmetric. 
By this we mean that the corresponding weigthed dual graph (which is homeomorphic to a line segment in the~cases (ii') and (iii')) labeled by the self-intersection numbers $-1$ and $-2$ does not admit non-trivial automorphisms.  
Therefore $Z(\mathbb{K}) \neq \varnothing$, and $Z$ is $\mathbb{K}$-rational. Thus $Z$ and $X$ are toric surfaces and a complement $D$ of a torus in $X$ is an anticanonical curve. So $(X,D)$ is a log canonical pair, and $\coreg_1(X) = 0$.

If $X$ has type (i') then $Z$ is a blow up of $\mathbb{P}^2$ at three geometric points lying on a line $L$. If $Q$ is any irreducible conic transversally intersecting $L$ then for $D = L + Q$ the pair $(\mathbb{P}^2, D)$ is log canonical. One can consider the crepant transform $D_Z$ of $D$ on $Z$, and see that
$$
\dim\mathcal{D}(Z, D_Z) = \dim\mathcal{D}(\mathbb{P}^2, D) = 1.
$$
Therefore by Lemma~\ref{lem-singweak} one has $\coreg_1(X) = \coreg_1(Z) = 0$.

If $X$ has type (iii) then $Z$ is a blow up of a Hirzebruch surface $\mathbb{F}_2$ at two geometric points lying on a fiber $F$ of the conic bundle $\mathbb{F}_2 \to \mathbb{P}^1$.
If $E$ is a $(-2)$-section and $C$ is any irreducible curve in~$|-K_X - E - F|$ then for $D = C + E + F$ the pair $(\mathbb{F}_2, D)$ is log canonical. One can consider the crepant transform $D_Z$ of $D$ on $Z$, and see that $\dim\mathcal{D}(Z, D_Z) = \dim\mathcal{D}(\mathbb{F}_2, D) = 1$. Therefore by Lemma~\ref{lem-singweak} one has $\coreg_1(X) = \coreg_1(Z) = 0$.

Now assume that $7 \leqslant d \leqslant 9$. One can blow up $Z$ at $d-6$ geometric points and obtain a weak del Pezzo surface $W$ of degree $6$ by Lemma~\ref{lem-blowup}. The anticanonical linear system $|-K_W|$ maps $W$ onto a surface $Y$ with at worst du Val singularities. We proved that $\coreg_1(Y) = 0$. Therefore $\coreg_1(X) = 0$, since by Lemmas~\ref{lem-singweak} and~\ref{lem-coregpush} one has
$$
\coreg_1(X) = \coreg_1(Z) \leqslant \coreg_1(W) = \coreg_1(Y) = 0.
$$

If $d = 7$ then $X$ is $\mathbb{K}$-rational, and one can blow up any $\mathbb{K}$-point on $Z$ not lying on $(-2)$-curves.

If $d = 8$ and $X \cong \mathbb{F}_1$ then $X$ is $\mathbb{K}$-rational, and we can apply the same reasoning as in the case~${d = 7}$. If $Z \cong \mathbb{F}_2$ then one can find a pair of geometric points defined over $\mathbb{K}$  not lying on a~fiber of the conic bundle $\mathbb{F}_2 \to B$, and not lying on the $(-2)$-curve. The blow up of this pair of~points is a del Pezzo surface of degree $6$.

If $\overline{X} \cong \mathbb{P}^1 \times \mathbb{P}^1$ then we can find two pairs $A_1$, $A_2$ and $B_1$, $B_2$ of the fibers of the projections onto the first and the second factor of $\mathbb{P}^1 \times \mathbb{P}^1$ respectively, such that the divisor $D = A_1 + B_1 + A_2 + B_2$ is~defined over $\mathbb{K}$. Therefore $\coreg_1(X) = 0$.

If $d = 9$ then $X$ is a Severi--Brauer surface, and there exists a triple of geometric points defined over $\mathbb{K}$ and  not lying on a line. The blow up of this triple is a del Pezzo surface of degree $6$. 
Arguing as above we see that $\mathrm{coreg}(X)=\mathrm{coreg}_1(X)=0$. This completes the proof.
\end{proof}

\begin{remark}
Note that in geometric case $G$-coregularity of del Pezzo surfaces of degree $9$ and $8$ can be greater than $0$ for certain groups $G$ (see \cite[Propositions~2 and~4.2.5]{LPT25}).
\end{remark}

\subsection{Cubic surfaces}

Now we pass to the case of del Pezzo surfaces of degree $3$ or greater with at~worst du Val singularities. Note that a del Pezzo surface of degree $3$ is a cubic surface in $\mathbb{P}^3$, and~therefore it is convenient to work with such surfaces. Moreover, we can apply Lemmas~\ref{lem-blowup} and~\ref{lem-coregpush} to reduce cases of del Pezzo surfaces of degree $4$ or greater to the case of cubic surfaces.

At first we want to study geometry of cubic surfaces. Therefore we can assume that the field~$\mathbb{K}$ is~algebraically closed.
Let $X$ be a cubic surface in $\mathbb{P}^3$. Then any element $D$ of $|-K_X|$ is an~intersection of $X$ and a plane. Coregularity $0$ can be achieved on the pair $(X, D)$ only in the following three cases:
\begin{itemize}
\item $D$ is a nodal cubic curve;
\item $D$ is a union of an irreducible conic and a line transversally intersecting this conic at two points;
\item $D$ is a triple of lines that do not have a common point.
\end{itemize}
Note that in all these cases $D$ contains at least one singular point. Denote such a point by $P$. Then either $P$ is singular on $X$, or $P$ is smooth on $X$ and $D = X \cap T_P X$, where $T_P X$ is the tangent plane of $X$ at $P$.

Conversely, for a smooth point $P \in X$, the intersection $D = X \cap T_P X$ is always a plane cubic curve singular at $P$. There are three other possibilities for $D$ which were not listed before:
\begin{itemize}
\item $D$ is a cuspidal cubic curve;
\item $D$ is a union of an irreducible conic and a line tangent to this conic;
\item $D$ is a triple of lines passing through a common point (in this case $D$ can contain a double or triple line).
\end{itemize}
In all these cases the pair $(X, D)$ is not log canonical. Note that these three cases exactly describe situations where $P$ is a \emph{parabolic point} of $X$ (see \cite[Section 1.1.5]{Dol12} for the definition). In~particular, for a smooth non-parabolic point $P$ on $X$ and $D = X \cap T_P X$ the pair $(X, D)$ is log canonical and the dual complex $\mathcal{D}(X, D)$ is one-dimensional. 

\begin{definition}
For a polynomial $F(x_1, \ldots, x_n)$ denote by $H(F)$ the \emph{Hessian matrix} of $F$:
$$
H(F) = \left( \dfrac{\partial^2 F}{\partial x_i \partial x_j} \right).
$$
The \emph{Hessian surface} $\operatorname{He}(X)$ of a hypersurface $X$ is defined by the equation $\operatorname{det} H(F) = 0$, where $F$ is a polynomial defining $X$, cf. \cite[Section 1.1.4]{Dol12}.
\end{definition}

From the results of \cite[Section 1.1.5]{Dol12} we deduce the following lemma.

\begin{lemma}
\label{lem-dens3itd}
Let $\mathbb{K}$ be a field of characteristic $0$, and let $X$ be a del Pezzo surface of degree $3$ with at worst du Val singularities defined over $\mathbb{K}$. Then there exists an open subset $U \subset X$, such that for any $P \in U$ the point $P$ is smooth on $X$, and for $D = X \cap T_P X$ the pair $(X, D)$ is log canonical and the dual complex $\mathcal{D}(X, D)$ is one-dimensional.
\end{lemma}

\begin{proof}
Put $U = X \setminus \operatorname{He}(X)$. Then any point of $U$ is smooth on $X$ by \cite[Proposition 1.1.19]{Dol12}. Moreover, any point of $U$ is not parabolic since all parabolic points are contained in $\operatorname{He}(X)$ by \cite[Theorem 1.1.20]{Dol12}. Therefore for any $P \in U$ and $D = X \cap T_P X$ the pair $(X, D)$ is log canonical and the dual complex $\mathcal{D}(X, D)$ is one-dimensional.
\end{proof}

Now we can prove the main result of this section.

\begin{theorem}
\label{thrm-main3itd}
Let $\mathbb{K}$ be a field of characteristic $0$, and let $X$ be a del Pezzo surface of degree $3$ or~greater with at worst du Val singularities defined over $\mathbb{K}$, such that $X(\mathbb{K})\neq\varnothing$. Then
$$
\coreg(X) = \mathrm{coreg}_1(X)=0.
$$
\end{theorem}

\begin{proof}
The surface $X$ is unirational by \cite[Theorems~29.4 and~30.1]{Man74}, since $X(\mathbb{K})\neq\varnothing$ and $\mathbb{K}$ is perfect and infinite. Therefore $X(\mathbb{K})$ is dense.

Assume that $(-K_X)^2 = 3$. Then for an open subset $U \subset X$ constructed in Lemma~\ref{lem-dens3itd} one has~${U(\mathbb{K}) \neq \varnothing}$ since $\mathbb{K}$ is infinite. Any $\mathbb{K}$-point $P$ in $U$ is smooth on $X$, and for~${D = X \cap T_P X}$ the pair~$(X, D)$ is log canonical and the dual complex $\mathcal{D}(X, D)$ is one-dimensional. Thus~${\coreg_1(X) = 0}$.

Now assume that $d = (-K_X)^2 > 3$. Consider a minimal resolution of singularities $Z \to X$, where~$Z$ is a weak del Pezzo surface. Let $R$ be a union of $(-2)$-curves on $X$. Then there exists \mbox{a $\mathbb{K}$-point $P$} in the open subset $U = X \setminus R$ since $X(\mathbb{K})$ is dense and $\mathbb{K}$ is infinite. One can blow up $Z$ at $P$ and obtain a weak del Pezzo surface of degree $d-1$ by Lemma~\ref{lem-blowup}. By repeating this procedure $d - 3$ times we obtain a weak del Pezzo surface $W$ of degree $3$. The anticanonical linear system $|-K_W|$ maps $W$ onto a cubic surface $Y$ with at worst du Val singularities. We have proved that $\coreg_1(Y) = 0$. Therefore $\coreg_1(X) = 0$, since by Lemmas~\ref{lem-singweak} and~\ref{lem-coregpush} one has
$$
\coreg_1(X) = \coreg_1(Z) \leqslant \coreg_1(W) = \coreg_1(Y) = 0.
$$
The proof is completed. 
\end{proof}

Now we can prove Theorem~\ref{main-thrm-3}.

\begin{proof}[{Proof of Theorem \ref{main-thrm-3}}]
Let $\mathbb{K}$ be a field of characteristic $0$, and let $X$ be a del Pezzo surface of~degree~$d$ with at worst du Val singularities defined over $\mathbb{K}$. If $d \geqslant 6$ then $\coreg_1(X) = 0$ by Proposition~\ref{prop-dp6789}.

If $d = 5$ then it is well-known that $X(\mathbb{K}) \neq \varnothing$. Therefore $\coreg_1(X) =0$ by Theorem~\ref{thrm-main3itd}. Moreover, if $d \in \{3, 4\}$ and $X(\mathbb{K}) \neq \varnothing$ then $\coreg_1(X) =0$ by Theorem~\ref{thrm-main3itd}. The proof is completed.
\end{proof}

In the following examples we show that the conditions of Theorem~\ref{thrm-main3itd} cannot be weakened. At~first we show that the claim of Theorem~\ref{thrm-main3itd} does not hold for pointless cubic surfaces.

\begin{example}
\label{ex-vikulova-2}
Let $\mathbb{K}$ be a field of characteristic $0$ containing a non-trivial root of unity of degree $3$. Assume that there exists a non-trivial Severi--Brauer surface $V$ over $\mathbb{K}$. Then by \cite[Lemma~3.4]{Vik24a} the~group $(\mathbb{Z} / 3\mathbb{Z})^2$ generated by elements $b$ and $c$ acts on $V$ by biregular automorphisms. Moreover, by \cite[Lemma 3.10]{Vik24a} the fixed points $p_1$, $p_2$, $p_3$ of $b$, and the fixed points $p_4$, $p_5$, $p_6$ of $c$ are in~general position on $\overline{V} = V \times_{\mathbb{K}} \overline{\mathbb{K}} \cong \mathbb{P}^2_{\overline{\mathbb{K}}}$.

Let $\mathbb{L}$ be a minimal extension of $\mathbb{K}$ such that all points $p_i$ are defined over $\mathbb{L}$. We claim that~${\operatorname{Gal}(\mathbb{L} / \mathbb{K}) \cong (\mathbb{Z} / 3\mathbb{Z})^2}$. Indeed, a point $p_1$ is defined over an extension $\mathbb{F} / \mathbb{K}$ of degree $3$. Therefore the points $p_2 = c(p_1)$ and $p_3 = c^2(p_1)$ are defined over $\mathbb{F}$. Similarly, each of the points $p_4$, $p_5$ and~$p_6$ is defined over an extension $\mathbb{F}' / \mathbb{K}$ of degree $3$.

Assume that $\mathbb{F} = \mathbb{F'}$. Then the line $L$ passing through the points $p_1$ and $p_4$, and the two lines $\operatorname{Gal}(\mathbb{F} / \mathbb{K})$-conjugate to $L$ pass through one point $P$ (one can check it by considering the regular action of $(\mathbb{Z} / 3\mathbb{Z})^2$ on $\mathbb{P}^2_{\overline{\mathbb{K}}} \cong \overline{V}$). Therefore the point $P$ is defined over $\mathbb{K}$ but $V(\mathbb{K}) = \varnothing$. From this contradiction it follows that $\mathbb{F} \neq \mathbb{F'}$, and all points $p_i$ are defined over an extension $\mathbb{L} / \mathbb{K}$ of degree~$9$.

Let $X \to V$ be the blow up of the six points $p_i$. Then $X$ is a pointless cubic surface such that all lines on $\overline{X}$ are defined over $\mathbb{L}$. Therefore any section of $X$ by a plane $H$ is either a smooth cubic curve, or a triple of conjugate lines, since other singular cubic curves always has a $\mathbb{K}$-point. But each of six $\operatorname{Gal}(\mathbb{L} / \mathbb{K})$-invariant triples of lines on $X_{\mathbb{L}}$ does not lie in one plane. Therefore any plane section of $X$ is a smooth cubic curve, and $\coreg_1(X) = 1$.
\end{example}

Now we consider fields of positive characteristic.

\begin{example}[{cf. \cite{Vik24b}}]
\label{ex-vikulova}
Consider a cubic del Pezzo surface $X$ defined over the finite field $\mathbb{F}_2$ given by the equation:
\[
x^2t + y^2z + yz^2 + xt^2 = 0.
\]
This surface is remarkable since it has $45$ Eckardt points defined over $\overline{\mathbb{F}}_2$ (in fact, over $\mathbb{F}_4$), which is the maximal possible number of such points, 
and $15$ Eckardt points defined over $\mathbb{F}_2$, which is again the maximal possible value.

We claim that $\mathrm{coreg}_1(X)=2$.
Indeed, there are exactly $15$ planes in $\mathbb{P}^3_{\mathbb{F}_2}$. Therefore for any plane~$H$ the intersection $D = X \cap H$ is a triple of lines passing through one Eckardt point, and the pair $(X, D)$ is not log canonical.

Moreover, over the field $\overline{\mathbb{F}}_2$ the surface $\overline{X} = X \times_{\mathbb{F}_2} \overline{\mathbb{F}}_2$ is isomorphic to the Fermat cubic
$$
x^3 + y^3 + z^3 + t^3 = 0.
$$
One can directly check that any section of $\overline{X}$ is either a smooth cubic curve, or a cuspidal cubic curve, or a union of an irreducible conic and a line tangent to this conic, or a triple of lines passing through one point. Therefore $\coreg_1(\overline{X}) = 1$.
\end{example}

Finally, let us look on del Pezzo surfaces of degree $2$.

\begin{example}
\label{ex-dp2-1}
Consider a del Pezzo surface $X$ of degree $2$ defined over $\mathbb{R}$ given by the equation
\[
x^4+y^4+z^4-w^2=0
\]
in $\mathbb{P}(1,1,1,2)$. Let $R\subset \mathbb{P}^2$ be the ramification curve of the double cover $X\to \mathbb{P}^2$. Note that~${X(\mathbb{R}) \neq \varnothing}$, but $R(\mathbb{R}) = \varnothing$.

We claim that $\mathrm{coreg}_1(X)=1$. Indeed, the elements $C\in|-K_X|$ are either irreducible, or~have two irreducible components (this follows from $C\cdot(-K_X)=2$). 
An element $C\in |-K_X|$ of coregularity $0$ is either a nodal irreducible curve, or a pair $(-1)$-curves which intersect transversally at two points. This would give us either an $\mathbb{R}$-point on $R$, or a pair of Galois-conjugate points on $R$. The latter case is excluded by a direct computation. Therefore $\coreg_1(X) = 1$. 
\end{example}

\begin{example}
\label{ex-dp2-2}
We can modify Example~\ref{ex-dp2-1}, and avoid the condition $R(\mathbb{K}) = \varnothing$.
Consider a del Pezzo surface $X$ defined over $\mathbb{Q}$ given by the equation
\[
x^4+y^4-z^4-w^2=0
\]
in $\mathbb{P}(1,1,1,2)$. Then $X(\mathbb{Q}) \neq \varnothing$, and $R(\mathbb{Q}) \neq \varnothing$. But there are only four $\mathbb{Q}$-points on $R$, and~tangent lines to these points correspond to pairs of tangent $(-1)$-curves on $X$. Moreover, by a direct computation one can check that $\coreg_1(X) = 1$.
\end{example}

\begin{remark}
In Example~\ref{ex-dp2-2} the curve $R$ contains only finite number of $\mathbb{K}$-points. Note that if~$R(\mathbb{K})$ is infinite and $R$ is smooth then one can argue as in the proof of Lemma~\ref{lem-dp2} and show that~${\coreg_1(X) = 0}$.
\end{remark}

\begin{remark}
Note that the surfaces constructed in Examples \ref{ex-vikulova-2}, \ref{ex-vikulova}, \ref{ex-dp2-1} and \ref{ex-dp2-2} for which coregularity is equal to $1$ or $2$ are special and have some extremal properties. For example, all these surfaces have large group of automorphisms. In particular, the automorphism groups of the cubic surfaces given in Example \ref{ex-vikulova} are maximal among smooth cubic surfaces defined over $\mathbb{F}_2$ and $\overline{\mathbb{F}}_2$ respectively.
\end{remark}

\section{Questions}
\label{sec-examples-and-questions}
In this section we assume that all varieties are defined over an algebraically closed field of characteristic $0$, although one can propose the same questions in a more general setting. 
We start with the following question which is natural to ask in view of Theorem \ref{cor-toric-models}.
\begin{question}
\label{question-looijenga}
Let $X$ be a klt del Pezzo surface. Can the condition $\mathrm{coreg}(X)=0$ or the existence of a toric model be characterized in terms of isotriviality of the family of anti-canonical curves on~$X$ (possibly, up to a bounded family of exceptions)? 
\end{question}

Another possible generalization of our results is to the case of higher dimensions. 
We start with the following question in this direction. 
\begin{question}
\label{question-isotrivial-fano}
Is there a smooth or mildly singular Fano threefold $X$ such that all smooth elements in $|-K_X|$ are isomorphic? What is $\mathrm{coreg}(X)$?
\end{question}

By \cite[Proposition 5.4]{Beau02}, the answer to Question \ref{question-isotrivial-fano} is negative in the case when $X$ is~smooth and $-K_X$ is very ample. For a general smooth Fano threefold such that $-K_X$ is not very ample, the~negative answer to Question \ref{question-isotrivial-fano} can be deduced from \cite{ALP24}, see e.g. the proof of \cite[Lemma~8.2]{ALP24} in the case of a general prime Fano threefold of degree $2$. We expect that the answer to Question \ref{question-isotrivial-fano} is negative for smooth Fano threefolds.  
On the other hand, it is easy to present an~example of a smooth Fano threefold which is birational to an isotrivial elliptic fibration.   
\begin{example}
Consider a smooth three-dimensional Veronese double cone $X$, that is a smooth Fano variety given by the equation of degree $6$ in $\mathbb{P}(1,1,1,2,3)$. In \cite[Lemma 5.7]{ALP24} it is shown that a general such $X$ has coregularity $0$. 
Note that $X$ admits a map $\phi\colon X\dashrightarrow \mathbb{P}^2$ given by the linear system $|-\frac{1}{2}K_X|$ whose general fiber is an elliptic curve. 
The map $\phi$ becomes a morphism after resolving the unique base-point of the linear system $|-\frac{1}{2}K_X|$. 
All smooth fibers of $\phi$ are isomorphic if and only if the equation of $X$ has the form:
\[
w^2=z^3+f(x_1,x_2,x_3),
\]
where $f$ is a homogeneous polynomial of degree $6$. In this case, the argument given in the proof of \cite[Lemma 5.7]{ALP24} does not apply, so we only have the bound $\mathrm{coreg}(X)\leqslant 1$. Indeed, if we pick two general elements $D_1,D_2\in|-\frac{1}{2}K_X|$, then the dual complex of $(X,D=D_1+D_2)$ is a~line segment. 
Note that a general $X$ is non-rational. It would be interesting to compute $\mathrm{coreg}(X)$ in~this case.
\end{example}

Note that according to \cite{ALP24}, the condition $\mathrm{reg}(X)=0$ can hold only for Fano threefolds with Picard rank $1$, or for the elements of the deformation families 2.1 and 10.1 in the notation of \cite{IP99}.

\begin{question}
Can Theorem \ref{cor-toric-models} be generalized to the three-dimensional case? In other words, can one relate the following  properties of a smooth or mildly singular rational Fano threefold: the~non-existence of a toric model; the condition that the coregularity is positive; the condition that if~$X$ is birational to an elliptic fibration $f\colon Y\to S$ then $f$ is isotrivial. 
\end{question}

The question when a Fano threefold is birational to an elliptic fibration is discussed in \cite{Ch07}.

The general principle is that the larger coregularity is, the more ``exceptional'' is the variety under consideration. In particular, such ``exceptional'' could have large automorphism groups. This observation is confirmed in several explicit examples, see Examples \ref{ex-vikulova-2}, \ref{ex-vikulova}, \ref{ex-dp2-1} and \ref{ex-dp2-2}. 
We~formulate the following question. 

\begin{question}
Assume that $X$ is a smooth del Pezzo surface defined over a field $\mathbb{K}$. Assume that $\mathrm{coreg}_1(X)=2$. Is it true that the automorphism group of $X$ is maximal among del Pezzo surfaces of given degree defined over $\mathbb{K}$?
\end{question}

We used the definition of coregularity which works for the varieties defined over a (not necessarily algebraically closed) field of characteristic $0$, see Definition \ref{defin-regularity} and Remark \ref{rem-non-closed-field}. It would be interesting to understand whether the same definition works in the case of a field of positive characteristic, see the discussion \cite[4.10]{Mo24}.  
We formulate a question regarding the coregularity in positive characteristic.

\begin{question}
\label{question-positive-char}
Is the notion of coregularity well defined in positive characteristic? 
Does Theorem~\ref{thm-1-2-complement} hold in positive characteristic?
\end{question}

In Section \ref{sec-dp-over-c} we used the theory of elliptic fibrations to compute the coregularity of du Val del Pezzo surfaces defined over $\mathbb{C}$. 
We expect that this approach 
could also work for du Val del Pezzo surfaces defined over a field of positive characteristic. Namely, the study of (quasi-)elliptic fibrations associated with a du Val del Pezzo surfaces was done in \cite{St21}. 
However, the explicit equations become more involved in positive characteristic. Hence, we formulate the following problem.

\begin{problem}
Сompute coregularity of du Val del Pezzo surfaces defined over a field of positive characteristic.
\end{problem}

\end{document}